\documentclass[11pt]{amsart}
\usepackage{graphicx}
\usepackage[USenglish,english]{babel}
\usepackage[margin=1.25in]{geometry}
\usepackage[T1]{fontenc}
\usepackage[latin1]{inputenc}
\usepackage{amsfonts}
\usepackage{amssymb}
\usepackage{amsthm}
\usepackage{amsmath}
\usepackage{fancyhdr}

\usepackage{float}
\usepackage{enumerate}
\usepackage{accents}
\usepackage{mathtools}
\usepackage{mathrsfs}
\usepackage{comment}
\usepackage{afterpage}
\usepackage{bbm}
\usepackage{thm-restate}
\usepackage{dsfont}
\usepackage{stmaryrd}
\usepackage{mleftright}
\usepackage{subfig}
\usepackage{faktor}
\usepackage{graphicx}
\usepackage{marvosym}
\usepackage[all]{xy}
\usepackage[colorlinks=true,allcolors=blue!80!black]{hyperref}
\usepackage{xcolor}

\newtheorem{theorem}{Theorem}
\numberwithin{theorem}{section}

\newtheorem{lemma}[theorem]{Lemma}

\newtheorem*{claim*}{Claim}

\newtheorem{proposition}[theorem]{Proposition}

\newtheorem{corollary}[theorem]{Corollary}

\newtheorem*{question*}{Question}

\theoremstyle{remark}
\newtheorem{remark}[theorem]{Remark}

\newtheorem*{remark*}{Remark}

\theoremstyle{definition}
\newtheorem{definition}[theorem]{Definition}

\newtheorem*{warning*}{Warning}
\newtheorem*{convention*}{Convention}
\newtheorem*{example*}{Example}

\newtheorem{rthm}{Theorem}
\newenvironment{reptheorem}[1]
  {\renewcommand\therthm{\ref*{#1}}\rthm}
  {\endrthm}

\newcommand{\R}{\mathbb{R}}
\newcommand{\B}{\mathbb{B}}

\newcommand{\Z}{\mathbb{Z}}

\newcommand{\Sa}{\mathbb{S}}
\newcommand{\Hy}{\mathbb{H}}
\newcommand{\D}{\mathbb{D}}

\newcommand{\C}{\mathcal{C}}

\newcommand{\calM}{\mathcal{M}}
\newcommand{\calK}{\mathcal{K}}
\newcommand{\calL}{\mathcal{L}}

\newcommand{\id}{\text{id}}

\newcommand{\overl}{\overline}

\newcommand{\interior}{\mathrm{int}}
\newcommand{\wt}{\widetilde}

\newcommand{\dinf}{\partial_\infty}

\newcommand{\diam}{\textrm{diam}}

\DeclareMathOperator{\Isom}{Isom}
\DeclareMathOperator{\Stab}{Stab}

\DeclareMathOperator{\Homeo}{Homeo}

\title[Generalized Cannon's conjecture]{On a generalization of Cannon's conjecture for cubulated hyperbolic groups}

\author{Corey Bregman}
\address{Department of Mathematics, Tufts University,  Medford, MA 02155 USA}
\email{corey.bregman@tufts.edu}
\urladdr{https://sites.google.com/view/cbregman} 

\author{Merlin Incerti-Medici}
\address{Universit\"at Wien, Fakult\"at f\"ur Mathematik, 1090 Wien, Austria}
\email{merlin.medici@gmail.com}
\urladdr{https://www.merlinmedici.ch/}


\begin{document}

\begin{abstract}
We show that cubulated hyperbolic groups with spherical boundary of dimension 3 or at least 5 are virtually fundamental groups of closed, orientable, aspherical manifolds, provided that there are sufficiently many quasi-convex, codimension-1 subgroups whose limit sets are locally flat subspheres. The proof is based on ideas used by Markovic in his work on Cannon's conjecture for cubulated hyperbolic groups with 2-sphere boundary. 
\end{abstract}

\maketitle

\setcounter{tocdepth}{1}
\tableofcontents

\section{Introduction}
The Cannon conjecture asserts that if $G$ is a (Gromov) hyperbolic group whose boundary $\dinf G$ is homeomorphic to $S^2$, then $G$ is commensurable up to finite quotients with a cocompact lattice $\Gamma$ in $ \Isom(\Hy^3)$, the isometry group of real hyperbolic 3-space.  More precisely,  if the Cannon conjecture is true, then there exists a cocompact lattice $\Gamma\leq \Isom(\Hy^3)$ and a short exact sequence \[1\to F\to G\to \Gamma\to 1\]
where $F$ is a finite group equal to the kernel of the action of $G$ on $\dinf G$.

Markovic \cite{Markovic13} showed that the Cannon conjecture holds under a further assumption, namely that $G$ has enough quasi-convex surface subgroups to separate pairs of points in $\dinf G$.  This assumption implies in particular that $G$ is cubulated, \emph{i.e.} $G$ acts geometrically on a CAT(0) cube complex. Under this action the hyperplane stabilizers are surface subgroups. Later Ha\"issinsky \cite{Haissinsky} strengthened this result by showing that if $G$ is cubulated, then $G$ has enough quasi-convex surface subgroups to separate pairs of points at infinity.  Conversely, since every cocompact lattice $\Gamma\leq \Isom(\Hy^3)$ is cubulated by results of Kahn--Markovic \cite{KahnMarkovic12} and  Bergeron--Wise \cite{BergeronWise12}, it follows that the Cannon conjecture is equivalent to the statement that every hyperbolic group $G$ with $\dinf G\cong S^2$ is cubulated. The goal of the present work is to prove the following extension of Markovic's result:

\begin{theorem}\label{thm:S3-cubulated}
    Let $G$ be a hyperbolic group such that $\dinf G\cong S^3$ and suppose $G$ contains quasi-convex, codimension-1 subgroups $\{H_1,\ldots, H_k\}$ satisfying
    \begin{enumerate}
    \item The $G$-translates of the limit sets $\dinf H_i$ separate pairs of points in $\dinf G$.
    \item $\dinf H_i$ is a locally flat 2-sphere in  $S^3\cong\dinf G $ for $i=1,\ldots, k$.  
    \end{enumerate}
Then there exists a finite-index, torsion-free subgroup $\hat{G}\leq G$ such that $\hat{G}\cong \pi_1(M)$, where $M$ is a closed, orientable $4$-manifold covered by $\R^4$.
\end{theorem}

We recall that an embedding $e : M \hookrightarrow N$ of a $k$-manifold into an $n$-manifold is {\it locally flat} if for every $x \in M$, there exists a neighbourhood $U$ of $e(x)$ in $N$ such that $(U, U \cap e(M))$ is homeomorphic to $(\mathbb{R}^n, \mathbb{R}^k)$. As in Markovic's case, hypothesis (1) implies that $G$ is cubulated with respect to a subset of $\{H_1,\ldots, H_k\}$. In particular, each $H_i$ is itself cubulated, hence by Markovic's result, each $H_i$ is virtually a lattice in $\Isom(\Hy^3)$. Since cubulated hyperbolic groups are residually finite, we may pass to a finite index subgroup of $G$ that is torsion-free and cubulated with respect to orientable, hyperbolic 3-manifold groups. Hypothesis (2) does not have a parallel in Markovic's setting because the Jordan curve theorem implies that the boundary of each surface subgroup is a locally flat circle in $S^2$. However, Apanasov--Tetenov \cite{ApanasovTetenov} have constructed examples of discrete subgroups in $\Isom(\Hy^4)$ that are isomorphic to closed hyperbolic 3-manifold groups but whose limit sets in $S^3$ are wild 2-spheres. As far as the authors are aware, such subgroups do not arise as quasi-convex subgroups of uniform lattices in $\Isom(\Hy^4)$, but the necessity of this assumption remains unclear.

The direct analog of Markovic's result is false when $\dinf G\cong S^3$. Indeed,  for all $n\geq 4$, Gromov--Thurston \cite{GromovThurston87} construct examples of closed, negatively curved $n$-manifolds that are not homotopy equivalent to real hyperbolic $n$-manifolds, but which were shown to be cubulated by Giralt \cite{Giralt17}. In fact, these manifolds contain enough \emph{totally geodesic} real hyperbolic $n$-manifolds to cubulate. 

The Gromov--Thurston examples underscore that for $n\geq 4$, if $G$ is a torsion-free hyperbolic group with $\dinf G\cong S^{n-1}$, the most one can hope for is that $G\cong \pi_1(M)$ for some closed aspherical $n$-manifold $M$. Remarkably, even without the cubulation assumption, Bartels--L\"uck--Weinberger \cite{BartelsLuckWeinberger10} show exactly this: if $G$ is a torsion-free hyperbolic group with $\dinf G\cong S^{n-1}$ for $n\geq 6$, then $G\cong \pi_1(M)$ for some closed aspherical $n$-manifold $M$. They remark that their proof ought to extend to the $n=5$ case as well, assuming certain surgery results hold. Therefore, Theorem \ref{thm:S3-cubulated} fills in a dimension where surgery techniques and the $s$-cobordism theorem are unavailable. Nevertheless, using results of Bartels--L\"uck \cite{BartelsLuck} on the Borel conjecture,  our proof works equally well for all $n\neq 5$, hence we will prove the main theorem in this generality:  

\begin{theorem}\label{thm:Sn-cubulated}
    Let $G$ be a torsion-free hyperbolic group such that $\dinf G\cong S^{n-1}$ and suppose $G$ contains quasi-convex, codimension-1 subgroups $\{H_1,\ldots, H_k\}$ satisfying:
    \begin{enumerate}
    \item The $G$-translates of the $\dinf H_i$ separate pairs of points in $\dinf G$.
    \item $\dinf H_i$ is a locally flat $(n-2)$-sphere in  $S^{n-1}\cong\dinf G $ for $i=1,\ldots, k$.  
    \end{enumerate}
    If $n\neq 5$, then there exists a finite-index, torsion-free subgroup $\hat{G}\leq G$ such that $\hat{G}\cong \pi_1(M)$, where $M$ is a closed orientable $n$-manifold covered by $\R^{n}$.
\end{theorem}

In the final section of the paper, we explore the local flatness condition in more depth. When $\dinf G=S^{n-1}$ and $H\leq G$ is a quasi-convex, codimension-1  subgroup such that $\dinf H=S^{n-2}$, we show
in Proposition \ref{prop:Flatness-Criterion} that $\dinf H$ is locally flat in $\dinf G$ if and only if both components of $\dinf G\setminus \dinf H$ are simply connected.

We then give two applications of this result, both of which may be regarded as generalizations of  analogous results for hyperbolic 3-manifolds.  A classical result, essentially due to Stallings \cite{Stallings61} asserts that if $V$ is the interior of a compact, orientable 3-manifold that is homotopy equivalent to an orientable surface $\Sigma$, then $M$ is homeomorphic to $\Sigma\times \R$. In particular, it follows that if $M$ is a closed, orientable hyperbolic 3-manifold and $V$ is the cover associated to a quasi-convex surface subgroup, then $V\cong \Sigma\times \R$ for some surface $\Sigma$.  In our setting, with $H\leq G$ as above, we say that $H$ is \emph{2-sided} if the action of $H$ on $\dinf G$ preserves each component of $\dinf G\setminus \dinf H$. We prove:

\begin{reptheorem}{thm:ProductCover}
Let  $M$ be a closed, orientable, aspherical $n$-manifold with $G=\pi_1(M)$ hyperbolic, $\dinf G\cong S^{n-1}$ and $n\geq 6$. Suppose $H\leq G$ is a quasi-convex, codimension-1, 2-sided subgroup such that $\dinf H\cong S^{n-2}$. If $C_H\rightarrow M$ denotes the cover associated to $H$,  then $\dinf H\subset \dinf G$ is locally flat if and only if there exists a closed, orientable $(n-1)$-manifold $N$ such that $C_H\cong N\times \R$.
\end{reptheorem}

Our second application further assumes that $G=\pi_1(M)$ is cubulated, and that $H\cong \pi_1(N)$ for some closed orientable manifold $N$. We prove that there exist finite covers of $\widehat M$ and $\widehat N$ of $M$ and $N$, respectively, such that $\widehat{N}$ embeds in $\widehat{M}$ as a submanifold:

\begin{reptheorem}{thm:Finite-Embedded-Cover}
Let  $M$ be a closed, orientable aspherical $n$-manifold with $G=\pi_1(M)$ cubulated hyperbolic and $n\geq 6$. Let $H\leq G$ be a quasi-convex, codimension-1 subgroup such that $H\cong \pi_1(N)$ for some closed aspherical $(n-1)$-manifold $N$. Then there exist finite covers $ \widehat{M}\rightarrow M$, $ \widehat{N}\rightarrow N$, and an embedding $\widehat{N}\hookrightarrow \widehat{M}$.
\end{reptheorem}
We regard this result as a generalization of the fact that a quasi-convex, $\pi_1$-injective surface subgroup of a hyperbolic 3-manifold group lifts to embedding in a finite cover. This is in turn a consequence of Agol's theorem \cite{Agol13} that cubulated hyperbolic groups are virtually special, and thus QCERF. However, in the case of surface subgroups of hyperbolic 3-manifold groups, one automatically has an immersed surface to try to lift. In the general case, we exploit a splitting theorem due to Cappell \cite{Cappell}, which allows us to find an embedded, 2-sided submanifold after passing to a cover. A similar technique (following Cappell's result in spirit) appeared in a paper of Kar--Niblo \cite{KarNiblo}.

\subsection{Overview of the paper} In Section \ref{sec:Preliminaries}, we review some results on quasi-convex subgroups of cubulated hyperbolic groups, the work of Bartels--L\"uck on the Borel conjecture for hyperbolic groups, 
and the topological generalized Schoenflies theorem of Mazur and Brown. 

In Section \ref{sec:G-complexes}, we introduce the concepts of generalized cell decompositions and $G$-complexes, first defined by Markovic. Although many of the results of this section are straightforward generalizations of the results in \cite{Markovic13}, we have provided many of the details for the convenience of the reader. Our exposition differs from Markovic in places, specifically with regard to condition $(\star)$ (see Remark \ref{rem:condition.star}), which is an attempt to recast some notions in purely topological terms. This paper arose in part from a desire to understand all of the ideas in \cite{Markovic13}, which we believe may be applicable more generally. We therefore hope that this alternative viewpoint will prove useful. Some of the more lengthy proofs have been relegated to the Appendix, when we felt that the techniques of the proof were not vital to understanding the remainder of the paper. 

In Section \ref{sec:Uniformizing-Sierpinski}, we review $n$-dimensional Sierpinski spaces and collect uniformization results about their embeddings in spheres. In \ref{sec:General-Markovic}, we prove Theorem \ref{thm:Sn-cubulated}.  Finally, in Section \ref{sec:local.flatness.of.codimension.one.submanifolds}, we discuss a condition which guarantees local flatness of boundaries of quasi-convex subgroups, and use it to prove the theorems about closed aspherical manifolds with cubulated hyperbolic fundamental group discussed above. 

\subsection{Acknowledgements} We would like to thank Vlad Markovic for several helpful comments. The first author was supported by NSF grant DMS-2401403. The second author was supported by the Austrian Science Fund (FWF) grant 10.55776/ESP124. For open access purposes, the authors have applied a CC BY public copyright license to any author-accepted manuscript version arising from this submission.




\section{Preliminary reductions}\label{sec:Preliminaries}
Suppose $G$ is cubulated hyperbolic. By results of Agol \cite{Agol13}, we know that $G$ is virtually special, hence virtually torsion-free. For simplicity, from now on we will  assume that $G$\emph{ is  a torsion-free hyperbolic group}. 

\subsection{Malnormal cubulations} 
Recall that if $G$ is a hyperbolic group and $H\leq G$ is a quasi-convex subgroup, then there is a closed, topological embedding $\dinf H\hookrightarrow \dinf G$.  If $\dinf G$ is connected, then $H$ has codimension-1 in $G$ if and only if 
$\dinf G\setminus \dinf H$ is disconnected.

\begin{definition}
Let $\mathcal{H}=\{H_1,\ldots, H_k\}$ be a collection of quasi-convex, codimension-1 subgroups of $G$. We say that $\mathcal{H}$ \emph{separates points at infinity} if for any two points $p\neq q\in\dinf G$, there exists a conjugate $H_i^g$ such that $p$ and $q$ lie in different connected components of $\dinf G\setminus \dinf H_i^g$.
\end{definition}

Bergeron--Wise \cite{BergeronWise12} prove that if $\mathcal{H}$ separates points at infinity  then there exists a finite-dimensional CAT(0) cube complex $X$ on which $G$ acts geometrically, and whose hyperplane stabilizers are conjugates of elements of $\mathcal{H}$. For this reason, if $\mathcal{H}=\{H_1,\ldots, H_k\}$ is a collection of quasi-convex, codimension-1 subgroups that separates points at infinity, we will say that $\mathcal{H}$ \emph{cubulates} $G$.

\begin{definition}
    A subgroup $H\leq G$ is called \emph{malnormal} if $H\cap H^g=\{1\}$ for every $g\notin H$. 
\end{definition}

Markovic showed that after passing to a cover, one can assume that each element of $\mathcal{H}$ is malnormal:

\begin{theorem}[Theorem 2.1, \cite{Markovic13}]\label{thm:Markovic-Malnormal} Suppose $G$ is hyperbolic and cubulated with respect to a collection of quasi-convex, codimension-1 subgroups $\mathcal{H}$.  Then there exists a finite-index subgroup $\widehat{G}\leq G$ that is cubulated with respect to a collection of malnormal, quasi-convex, codimension-1 subgroups.
\end{theorem}

\subsection{The Borel conjecture for hyperbolic groups}
Let $G$ be a finitely generated group and let $M$ be a closed, aspherical $M$ manifold such that $G\cong \pi_1(M)$. In particular, $M$ is a finite-dimensional $K(G,1)$, and therefore $G$ is torsion-free. The Borel conjecture asks whether $M$ is unique up to homeomorphism. More precisely, if $N$ is another $n$-manifold and $f\colon M\rightarrow N$ is a homotopy equivalence, is $f$ necessarily homotopic to a homeomorphism?

The Borel conjecture holds in low dimensions as a consequence of the classification of manifolds of dimensions $n\leq 3$. In dimension 1, the circle is the only closed 1-manifold.  In dimension 2, it follows from classification of closed surfaces.  In dimension 3, it is now a corollary of the Geometrization Theorem due to Perelman \cite{Perelman02, Perelman03a, Perelman03b}, relying on previous work of Waldhausen, for Haken 3-manifolds, Scott \cite{Scott83} for Seifert-fibered 3-manifolds, and Mostow rigidity for hyperbolic 3-manifolds (earlier, Gabai--Meyerhoff--Thurston \cite{GabaiMeyerhoffThurston03} had shown that a homotopy hyperbolic 3-manifold is indeed homeomorphic to a hyperbolic 3-manifold).  In high dimensions $n\geq 5$, Farrell--Jones \cite{FarrellJones89b, FarrellJones89a} proved the Borel conjecture when $M$ admits a metric of non-positive curvature. However, they also proved that the analogous conjecture is false in the smooth category \cite{FarrellJonesExotic}. 

The Borel conjecture has been solved in high dimensions for torsion-free hyperbolic groups by work of Bartels--L\"uck \cite{BartelsLuck}. In fact, they prove it for a much more general class of groups, but we will only need the following special case.

\begin{theorem}[Bartels--L\"uck]\label{thm:Hyperbolic-Borel-Conjecture}
    Let $M$ be a closed, aspherical manifold of dimension $n\geq 5$ such that $\pi_1(M)$ is hyperbolic.  If $N$ is an $n$-manifold and  $f\colon M\rightarrow N$ is a homotopy equivalence, then $f$ is homotopic to a homeomorphism.
\end{theorem}
Combining this with the low-dimensional results mentioned above we have:
\begin{corollary}\label{cor:Hyperbolic-Borel-Except-4}
    Let $M$ be a closed, aspherical manifold of dimension $n\neq 4$ such that $\pi_1(M)$ is hyperbolic. If $N$ is an $n$-manifold and $f\colon M\rightarrow N$ is a homotopy equivalence, then $f$ is homotopic to a homeomorphism.
\end{corollary}

\subsection{Local flatness and the generalized Schoenflies theorem}

In contrast to the low-dimensional case where $\partial_{\infty} G \cong S^2$ the properties of the embedding of $\partial_{\infty} H_i \hookrightarrow \partial_{\infty} G$ become more delicate in higher dimensions. To deal with this, we need some terminology and results from the theory of embeddings.

\begin{definition} \label{def:locally.flat}
Let $k < n$, $M^k$ and $N^n$ a $k$- and $n$-dimensional manifold respectively and let $x \in M$. An embedding $e : M \hookrightarrow N$ is called {\it locally flat at $x$}, if there exists a neighborhood $U$ of $e(x)$ in $N$ such that $(U, U \cap e(M))$ is homeomorphic to $(\mathbb{R}^n, \mathbb{R}^k)$. We say that $e$ is {\it locally flat}, if it is locally flat at every point in $M$.
\end{definition}

\begin{definition} \label{def:tame.ball}
    Let $N^n$ be an $n$-dimensional manifold, $k \leq n$, and $D^k$ the closed unit ball in $\mathbb{R}^k$. We call an embedding $e : D^k \hookrightarrow N$ a {\it tame $k$-ball}, if the restriction of $e$ to $\partial D^k$ is a locally flat embedding.
\end{definition}

\begin{definition} \label{def:standard.embedding}
    Let $S^{n}=\{(x_0,\ldots,x_n)\mid \sum_{i=0}^nx_i^2=1\}\subset \R^{n+1}$ be the $n$-sphere. For $k\leq n$, the \emph{standard} $S^k$ in $S^n$ is the subspace formed as the intersection of $S^n$ with each of the hyperplanes $H_i=\{x_i=0\}$ for $k+1\leq i\leq n$.
\end{definition}

The generalized Schoenflies theorem due to Mazur \cite{Mazur} and Brown \cite{Brown} relates the notion of local flatness with standard embeddings in codimension 1.

\begin{theorem}[Generalized Schoenflies theorem] \label{thm:generalized.Schoenflies}
    Let $e : S^{n-1} \hookrightarrow S^n$ be a locally flat embedding and let $Z$ be one of the two connected components of $S^n \setminus e(S^{n-1})$. Then there exists a homeomorphism $(\overline{Z}, Z) \rightarrow (D^n, \partial D^n)$.
\end{theorem}

\begin{remark} \label{rem:schoenflies.implies.standard}
    The generalized Schoenflies theorem implies that locally flat embeddings $S^{n-1} \hookrightarrow S^n$ are conjugate by a homeomorphism to the standard embedding. Indeed, let $Z^{+}$ and $Z^{-}$ be the two connected components of $S^n \setminus e(S^{n-1})$. By the generalized Schoenflies theorem, both $\overline{Z^{+}}$ and $\overline{Z^{-}}$ are homeomorphic to the closed disk $D^{n}\subset\mathbb{R}^{n}$. In particular the glueing $\overline{Z^{+}} \coprod_{\Lambda} \overline{Z^{-}}$ is homeomorphic to the glueing $D^{n} \coprod_{S^{n-1}} D^{n}$ that produces an $n$-dimensional sphere with a standard $S^{n-1}$ inside. Therefore, the homeomorphisms $\overline{Z^{\pm}} \rightarrow D^{n}$ allow us to construct a homeomorphism $(\partial_{\infty} G, \Lambda) \rightarrow (S^{n}, S^{n-1})$, where $S^{n-1}\hookrightarrow S^{n}$ is standard.
\end{remark}




\section{Review of $G$-complexes}\label{sec:G-complexes}

For $n \geq 3$, let $D^{n}$ denote the closed unit ball in $\mathbb{R}^{n}$ and $S^{n-1} = \partial D^{n}$ its boundary. We denote the interior by $\interior(D^{n})$. In this section, we review and generalize the notion of $G$-complexes used in \cite{Markovic13}. Our main-objective is to show that all the key results which Markovic proved for $G$-complexes on the pair $(D^3, S^2)$ hold for the pair $(D^{n}, S^{n-1})$ as well. For the convenience of the reader, we have reproduced many of the proofs here or in Appendix \ref{sec:Appendix}. We feel that some of the ideas and techniques herein may be useful more generally; for that reason we have erred on the side of providing all the details in our exposition. However, when the proofs truly do not differ significantly from those of \cite{Markovic13} apart from superficial changes, we will simply refer to that paper.  

Before we start, we deal with some basic notation, definitions, and conventions. Let $X$ be a compact metric space, $C \subset X$ and $a \in X$. Throughout this section, we use the notation
\[ d(C, x) := \sup \{ d(c,x) \vert c \in C \}. \]

We will use the following terminology from \cite{KarrerMiraftabZbinden24}.

\begin{definition}
    Let $X$ be as above and $(U_m)_m$ a sequence of subsets. We define the limit $\lim_{m \rightarrow \infty} U_m$ to be the collection of points obtained as limits of convergent sequences $(x_m)_m$, where $x_m \in U_m$ for every $m$.
\end{definition}

Note that if $X$ is a compact metric space, then a sequence $(U_m)_m$ converges to a singleton set $\{ p \}$, $p \in X$ if and only if for every neighbourhood $U$ of $p$, there exists $M$ such that $U_m \subset U$ for all $m \geq M$.

Let $H < \Homeo(X)$ be a subgroup. We say that $H$ is a {\it convergence group} if for every sequence $(h_m)_m$ of distinct $h_m$ in $G$, there exists a subsequence $(h_{m_k})_k$ and points $a, b \in X$ such that $h_{m_k} \rightarrow a$ uniformly on compact subsets of $X \setminus \{ b \}$. (We denote the constant map that sends all of $X$ to $a$ by $a$ as well.) Note that in this case the sequence $(h_{n_k}^{-1})_k$ converges to $b$ uniformly on compact subsets of $X \setminus \{ a \}$.

Let $G$ be a group. A $G$-action on $X$ is a monomorphism $G \rightarrow \Homeo(X)$ (all our actions are faithful). Regarding group multiplication, we use the convention that $\mu(gh) = \mu(h) \circ \mu(g)$.


\subsection{Generalized cell decompositions and $G$-complexes} \label{subsec:generalized.cell.decompositions.and.G-complexes}

\begin{definition} \label{def:generalized.cell.decomposition}
    Let $K \subset D^{n}$ be a closed subset 
    and let $\mathcal{U}$ denote the collection of connected components of $D^{n} \setminus K$. We say the pair $(K, \mathcal{U})$ is a {\it generalized cell decomposition of $D^{n}$} if the following holds:
    \begin{itemize}
        \item $S^{n-1} \subset K$.
        
        \item Every open set $U \in \mathcal{U}$ is homeomorphic to $\interior(D^{n})$ and its boundary $\partial U$ is homeomorphic to $S^{n-1}$.

        \item Every sequence $(U_m)_m$ of distinct elements in $\mathcal{U}$ has a subsequence $(U_{m_k})_k$ such that $\lim_{k \rightarrow \infty} U_{m_k}$ consists of a single point.

    \end{itemize}
\end{definition}

\begin{remark} \label{rem:condition.star}
    Our definition of a generalized cell decomposition differs from the one given in \cite{Markovic13}. Instead of our third condition, Markovic requires the following:
    \begin{itemize}
        \item[$(\star)$] For every $\delta > 0$, there exists $N(\delta) \in \mathbb{N}$ such that there are at most $N(\delta)$ many elements in $\mathcal{U}'$ that have diameter greater than $\delta$.
    \end{itemize}
    It is an easy exercise to see that this is equivalent to our third condition (using the fact that $D^n$ is compact). Therefore, this is an equivalent definition to the one given in \cite{Markovic13} and we exploit both $(\star)$ and the third property in our definition. For technical reasons that will become apparent later in this section, we prefer the third condition to be expressed in purely topological terms.
    
\end{remark}

The following two elementary examples of generalized cell decompositions will be useful. First, we may consider $K = S^{n-1}$. Then $\mathcal{U}$ consists of one element, the set $\interior(D^{n})$. Second, we may consider $K = D^{n}$. In this case, $\mathcal{U}$ is the empty set. These two examples are the minimal and maximal generalized cell decompositions with respect to the following partial order.

\begin{definition} \label{def:refinement}
    Let $(K, \mathcal{U})$, $(K', \mathcal{U}')$ be two generalized cell decompositions of $D^{n}$. We say that $(K, \mathcal{U})$ is a {\it refinement} of $(K', \mathcal{U}')$ if $K' \subset K$. This defines a partial order on the collection of all generalized cell decompositions of $D^{n}$.
\end{definition}

We now introduce the notion of an action of a group $G$ on a generalized cell decomposition.

\begin{definition} \label{def:G-action}
    Let $G$ be a group and $(K, \mathcal{U})$ a generalized cell decomposition of $D^{n}$. A {$G$-action} on $(K, \mathcal{U})$ is a $G$-action $\mu : G \rightarrow \Homeo(K)$ such that
    \begin{enumerate}
        \item For every $g \in G$, $\mu(g)(S^{n-1}) = S^{n-1}$.

        \item For all $g \in G$, $U \in \mathcal{U}$, there exists $U' \in \mathcal{U}$ such that $\mu(g)(\partial U) = \partial U'$. 
    \end{enumerate}
\end{definition}

Note that every $G$-action on $(K, \mathcal{U})$ induces a $G$-action on $S^{n-1}$. We will assume for all our $G$-actions that they act by orientation preserving homeomorphisms on $S^{n-1}$.\\

A $G$-action induces an action on the collection of `cells' $\mathcal{U}$. This follows from the following Lemma.

\begin{lemma} \label{lem:action.on.cells.well-defined}
    Let $V_1, V_2 \in \mathcal{U}$ such that $\partial V_1 = \partial V_2$. Then $V_1 = V_2$.
\end{lemma}

In particular, if $\mu$ is a $G$-action on a generalized cell decomposition $(K, \mathcal{U})$, then for every $g \in G$, $U \in \mathcal{U}$, there exists a unique $U' \in \mathcal{U}$ such that $\mu(g)(\partial U) = \partial U'$. We write
\[ \mu(g)(U) = U', \]
which defines an action of $G$ on $\mathcal{U}$.

\begin{proof}
    Let $V_1, V_2 \in \mathcal{U}$ such that $\partial V_1 = \partial V_2$. Since $V_1, V_2$ are connected components of $D^{n} \setminus K$, they are either disjoint or equal. Suppose they are disjoint. Since $V_1, V_2$ are open subsets of $D^{n}$ such that $(\overline{V_1}, \partial V_1)$ and $(\overline{V_2}, \partial V_2)$ are homeomorphic to $(D^{n}, S^{n-1})$, the union $\overline{V_1} \cup \overline{V_2}$ is homeomorphic to the closed manifold $D^{n} \coprod_{S^{n-1}} D^{n} \approx S^{n}$, where the left-hand-side denotes a glueing of two closed unit balls along a homeomorphism between their boundaries. Since $\overline{V_1} \cup \overline{V_2}$ is compact, it is an embedded copy of $S^n$ in $D^n$. However, closed $n$-dimensional manifolds cannot be embedded into $\mathbb{R}^{n}$ (see for example {\cite[Corollary 2B.4]{Hatcher02}}). We obtain a contradiction and conclude that $V_1 = V_2$.
\end{proof}

\begin{definition} \label{def:free.action}

    For $U \in \mathcal{U}$, we define its {\it stabilizer} to be
    \[ \Stab_{\mu}(U) := \{ g \in G \vert \mu(g)(\partial U) = \partial U \}. \]
\end{definition}

\begin{definition} \label{def:G-complex}
    Let $G$ be a group, $(K, \mathcal{U})$ a generalized cell decomposition of $D^{n}$, and $\mu$ a $G$-action on $(K, \mathcal{U})$. We call the triple $(\mu, K, \mathcal{U})$ a {\it $G$-complex}, if for every $U \in \mathcal{U}$ there exists a homeomorphism $\Phi^{U} : (D^{n}, S^{n-1}) \rightarrow (\overline{U}, \partial U)$ such that the following two conditions are satisfied:
    \begin{enumerate}
        \item $\Phi^{U}$ fixes every point in $S^{n-1} \cap \partial U$.

        \item The collection $\{ \Phi^{U} \}_{U \in \mathcal{U}}$ is $G$-equivariant on the boundary, that is for every $g \in G$,
        \[ \mu(g) \circ \Phi^{U} = \Phi^{\mu(g)(U)} \circ \mu(g) \quad \text{on $S^{n-1}$}. \]
    \end{enumerate}

    If we remember the choices of $\Phi^{U}$, we call the collection $(\mu, K, \mathcal{U}, \Phi^{\mathcal{U}} )$ a {\it marked $G$-complex}.
\end{definition}

\begin{definition} \label{def:free.and.convergence.G-complex}
Let $\mu : G \rightarrow \Homeo(K)$ be a $G$-action on $(K, \mathcal{U})$. 
\begin{itemize}
    \item We say that $\mu$ is {\it free} if for every $g \in G \setminus \{ 1 \}$, $\mu(g)$ has no fixed points in $K \cap \interior(D^{n})$.

    \item A $G$-complex $(\mu, K, \mathcal{U})$ is called {\it free} if and only if the $G$-action $\mu$ is free.

    \item We say $(\mu, K, \mathcal{U})$ is a {\it convergence $G$-complex} if and only if $\mu(G)$ acts as a convergence group on $K$.

\end{itemize}
    

\end{definition}

\begin{remark} \label{rem:convergence.fixed.points.are.in.boundary}
    Let $(g_m)_m$ be a sequence of distinct elements in $G$. If $(\mu, K, \mathcal{U})$ is a convergence $G$-complex, then by definition we find two points $a, b \in K$ and a subsequence $(g_{m_k})_k$ such that $(\mu(g_{m_k}))_k \rightarrow a$ uniformly on $K \setminus \{ b \}$ and $(\mu(g_{m_k}^{-1}))_k \rightarrow b$ uniformly on $K \setminus \{ a \}$. Since $\mu(G)$ preserves $S^{n-1} \subset K$, we conclude that $a, b \in S^{n-1}$. This holds for any sequence of distinct elements in $G$.
\end{remark}

\begin{remark} \label{rem:torsion.free.convergence.implies.free}
    If $G$ is torsion free, then any convergence $G$-complex is also free for the following reason: Suppose $(\mu, K, \mathcal{U})$ is a convergence $G$-complex and let $g \in G \setminus \{ 1 \}$. Since $G$ has no torsion, $(g^m)_m$ is a sequence of distinct elements in $G$ to which we can apply the convergence-property. We conclude that $\mu(g)$ cannot have any fixed points in $K \cap \interior(D^n)$.
    

\end{remark}

\begin{definition}
    Let $(\mu, K,\mathcal{U})$ be a $G$-complex with marking $\Phi^\mathcal{U}$ and let $f\colon D^n\rightarrow D^n$ be a homeomorphism. Define $f_*(\mu)(g) := f\circ\mu(g)\circ f^{-1}$ for all $g\in G$, and $f(\mathcal{U}) := \{f(U)\mid U\in \mathcal{U}\}$. The \emph{pushforward} of $(\mu, K,\mathcal{U})$ is
    \[f_*(\mu, K,\mathcal{U}) := (f_*(\mu),f(K),f(\mathcal{U}))\]
    with marking $f_*(\Phi^\mathcal{U}) := \{\Phi^{f(U)}=f\circ\Phi^U \circ f^{-1}\mid f(U)\in f(\mathcal{U}) \}$.
\end{definition}
That $f_*(\mu,K,\mathcal{U})$ satisfies (1) and (2) of Definition \ref{def:G-complex} is an exercise, noting that $f$ is actually a homeomorphism of pairs $(D^{n},\partial D^{n-1})\smash{\xrightarrow{\simeq}}(D^{n},\partial D^{n-1})$. Similarly, one can define a pullback of a $G$-complex, but these are completely interchangeable for our purposes.
\begin{proposition}
    Being a convergence $G$-complex or a free $G$-complex is preserved under pushforward. 
\end{proposition}
\begin{proof}
    Both of these are defined by properties of the $G$-action in the compact-open topology and the fixed points of elements of $G$ in the boundary $S^{n-1}$. They are both preserved by pushforward since $f$ is a homeomorphism of pairs $(D^{n},\partial D^{n-1})\xrightarrow{\simeq}(D^{n},\partial D^{n-1})$.
\end{proof}

\begin{proposition} \label{prop:quotient.of.maximal.G-action}
    Set $K_0 := D^{n}$ and let $(\mu, K_0, \emptyset)$ be a free convergence $G$-complex. Then the quotient $M := \faktor{\interior(D^{n})}{\mu(G)}$ is an aspherical $n$-dimensional manifold whose fundamental group is isomorphic to $G$.
\end{proposition}

\begin{proof}
    Since $(\mu, K_0, \emptyset)$ is free, $\mu(g)$ has no fixed point on $K_0 \cap \interior(D^{n}) = \interior(D^{n})$ for every $g \in G \setminus \{ 1 \}$. In other words, the action of $\mu(G)$ on $\interior(D^{n})$ is free. Since $(\mu, K_0, \emptyset)$ is a convergence $G$-complex, $\mu(G)$ is a convergence group on $D^{n}$ and thus its action on $\interior(D^{n})$ is proper (w.\,r.\,t.\,the euclidean metric). Therefore, the quotient $\faktor{ \interior(D^{n}) }{\mu(G)}$ is an $n$-dimensional manifold. Since $\interior(D^{n})$ is contractible, this quotient is aspherical and its fundamental group is isomorphic to $G$.
\end{proof}

We see that Theorems \ref{thm:S3-cubulated} and \ref{thm:Sn-cubulated} follow if we are able to construct a free convergence $G$-complex $(\mu, D^{n}, \emptyset)$ from the assumptions made about $G$ in these theorems.

\subsection{Refinements of $G$-complexes} \label{subsec:refinements.of.G-complexes}

\begin{definition} \label{def:refinement.of.complex}
    Let $(\mu, K, \mathcal{U})$ and $(\mu', K', \mathcal{U}')$ be two $G$-complexes on $D^n$. We say $(\mu, K, \mathcal{U})$ is a {\it refinement} of $(\mu', K', \mathcal{U}')$ if $K' \subset K$ and $\mu(g) = \mu'(g)$ on $K'$.
\end{definition}

The notion of refinement defines a partial order on the collection of all $G$-complexes. Note that if $(\mu, K, \mathcal{U})$ is a refinement of $(\mu', K', \mathcal{U}')$, then $(K, \mathcal{U})$ is a refinement of $(K', \mathcal{U}')$. 

The next two lemmas are key technical results concerning refinements of $G$-complexes, but we feel their proofs are lengthy and not essential for understanding the remainder of the paper. In order to avoid interrupting the flow of this section, we have relegated their proofs to Appendix \ref{sec:Appendix}.  The first concerns the relationship between stabilizers under refinement, and is stated without proof in \cite[p. 1047]{Markovic13}:

\begin{lemma} \label{lem:refinements.and.stabilizers}
    Let $(\mu, K, \mathcal{U})$ be a refinement of $(\mu', K', \mathcal{U}')$ and suppose $U \subset U'$, where $U \in \mathcal{U}$ and $U' \in \mathcal{U}'$. Then $\Stab_{\mu}(U) < \Stab_{\mu'}(U')$.
\end{lemma}

The second key result concerns the behavior of convergence $G$-complexes under refinement. The proof of this lemma is very similar to the proof in the $D^3$-case, but due to its technical nature we have reproduced it in full to make sure it does not rely on dimension 3.  


\begin{lemma} \label{lem:inheritance.of.convergence.property}
    Let $(\mu, K, \mathcal{U})$ be a $G$-complex which is a refinement of a convergence $G$-complex $(\mu', K', \mathcal{U}')$. 
    If for every $U' \in \mathcal{U}'$, the group $\mu(\Stab_{\mu'}(U'))$ is a convergence group on $\overline{U'} \cap K$, then $(\mu, K, \mathcal{U})$ is a convergence $G$-complex as well.
\end{lemma}

\subsection{Refining one $G$-complex by another} \label{subsec:refining.one.G-complex.by.another}

Consider two marked $G$-complexes $(\mu_1, K_1, \mathcal{U}_1, \Phi_1^{\mathcal{U}_1} )$ and $(\mu_2, K_2, \mathcal{U}_2, \Phi_2^{\mathcal{U}_2} )$ that agree on the boundary, i.e. suppose that for all $g \in G$,
\[ \mu_1(g)\vert_{S^{n-1}} = \mu_2(g)\vert_{S^{n-1}}. \]
Markovic \cite{Markovic13} constructs the refinement of $(\mu_1, K_1, \mathcal{U}_1, \Phi_1^{\mathcal{U}_1} )$ induced by $(\mu_2, K_2, \mathcal{U}_2, \Phi_2^{\mathcal{U}_2} )$ in the $D^3$-case. The construction works the same way in the $D^{n}$-case and it goes as follows. We define
\[ \mathcal{U} := \{ \Phi_1^{V}(W) \vert V \in \mathcal{U}_1, W \in \mathcal{U}_2 \}, \]

\[ K := K_1 \cup \bigcup_{V \in \mathcal{U}_1} \Phi_1^{V}(K_2). \]
For every $U = \Phi_1^{V}(W) \in \mathcal{U}$, the map $\Phi^{U} := \Phi_1^V \circ \Phi_2^{W} : (D^{n}, S^{n-1}) \rightarrow (\overline{U}, \partial U)$ is a homeomorphism and we define $\Phi^{U}$ as the marking for $U \in \mathcal{U}$. One easily checks that $K$ is the complement of the union of all elements in $\mathcal{U}$ and thus we obtain that $(K, \mathcal{U})$ satisfies the first two conditions of a generalized cell decomposition of $D^{n}$ (see figure \ref{fig:refinement}).

\begin{figure}
    \centering
    \def\svgwidth{5in}
\begingroup%
  \makeatletter%
  \providecommand\color[2][]{%
    \errmessage{(Inkscape) Color is used for the text in Inkscape, but the package 'color.sty' is not loaded}%
    \renewcommand\color[2][]{}%
  }%
  \providecommand\transparent[1]{%
    \errmessage{(Inkscape) Transparency is used (non-zero) for the text in Inkscape, but the package 'transparent.sty' is not loaded}%
    \renewcommand\transparent[1]{}%
  }%
  \providecommand\rotatebox[2]{#2}%
  \newcommand*\fsize{\dimexpr\f@size pt\relax}%
  \newcommand*\lineheight[1]{\fontsize{\fsize}{#1\fsize}\selectfont}%
  \ifx\svgwidth\undefined%
    \setlength{\unitlength}{566.43695572bp}%
    \ifx\svgscale\undefined%
      \relax%
    \else%
      \setlength{\unitlength}{\unitlength * \real{\svgscale}}%
    \fi%
  \else%
    \setlength{\unitlength}{\svgwidth}%
  \fi%
  \global\let\svgwidth\undefined%
  \global\let\svgscale\undefined%
  \makeatother%
  \begin{picture}(1,0.43853138)%
    \lineheight{1}%
    \setlength\tabcolsep{0pt}%
    \put(0,0){\includegraphics[width=\unitlength,page=1]{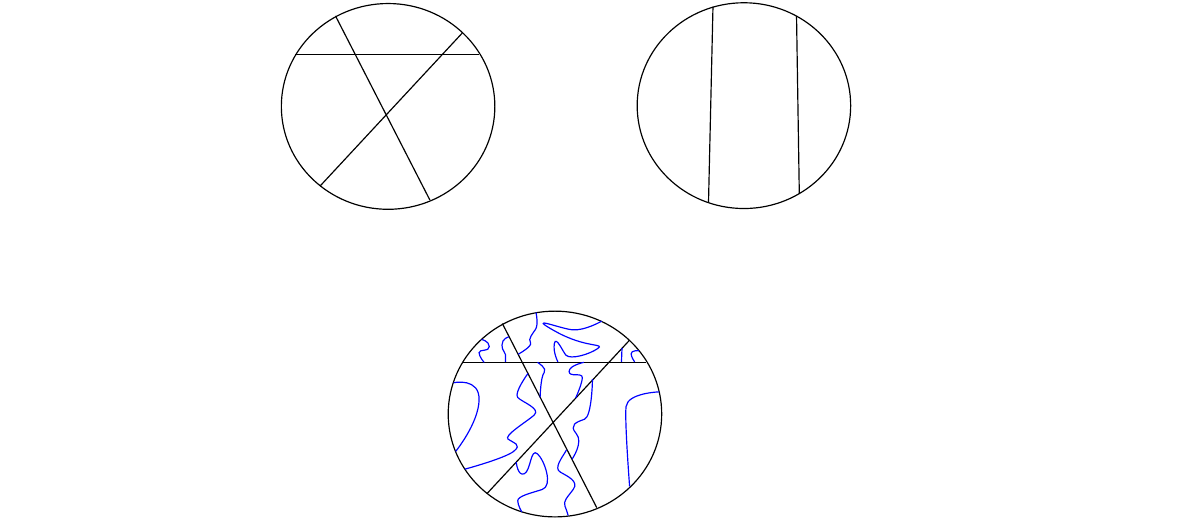}}%
    \put(0.39209356,0.42444551){\color[rgb]{0,0,0}\makebox(0,0)[lt]{\lineheight{1.25}\smash{\begin{tabular}[t]{l}$(\mu_1, K_1, \mathcal{U}_1)$\end{tabular}}}}%
    \put(0.70715399,0.42005748){\color[rgb]{0,0,0}\makebox(0,0)[lt]{\lineheight{1.25}\smash{\begin{tabular}[t]{l}$(\mu_2, K_2, \mathcal{U}_2)$\end{tabular}}}}%
    \put(0.56059388,0.15238745){\color[rgb]{0,0,0}\makebox(0,0)[lt]{\lineheight{1.25}\smash{\begin{tabular}[t]{l}$(\mu, K, \mathcal{U})$\end{tabular}}}}%
    \put(0,0){\includegraphics[width=\unitlength,page=2]{Refinement.pdf}}%
  \end{picture}%
\endgroup%

    \caption{This figure shows the refinement of $(\mu_1, K_1, \mathcal{U}_1)$ induced by $(\mu_2, K_2, \mathcal{U}_2)$. The blue lines in the bottom figure are the parts of the refinement induced by the second refinement.}
    \label{fig:refinement}
\end{figure}

To conclude that $(K, \mathcal{U})$ is a generalized cell decomposition, we are left to show that every sequence of distinct elements $(U_m)_m$ in $\mathcal{U}$ has a subsequence that converges to a point. Consider a sequence of distinct elements $(U_m)_m$ and suppose $U_m = \Phi_1^{V_m}(W_m)$ for $V_m \in \mathcal{U}_1$, $W_m \in \mathcal{U}_2$. There are two cases to consider.

If the sequence $(V_m)_m$ contains infinitely many distinct elements, we may choose a subsequence $(U_{m_k})_k$ such that all $V_{m_k}$ are pairwise distinct. Since $(K_1, \mathcal{U}_1)$ is a generalized cell decomposition, $(V_{m_k})_k$ admits a subsequence (which we again denote by $(V_{m_k})_k$ again) that converges to a point. Since $U_{m_k} \subset V_{m_k}$, the sequence $(U_{m_k})_k$ converges to that same point.

Now suppose the sequence $(V_m)_m$ contains only finitely many distinct elements. Then there exists some $V_M$ that $U_m \subset V_M$ for infinitely many $m$. Consider the subsequence $(U_{m_k})_k$ of all elements $U_{m_k} \subset V_M$. Since all $U_{m_k}$ are pairwise distinct, we obtain a sequence $( (\Phi_1^{V_M})^{-1}(U_{m_k}))_k$ of pairwise distinct elements of $\mathcal{U}_2$. Since $(K_2, \mathcal{U}_2)$ is a generalized cell decomposition, this sequence contains a subsequence (which we again denote by $( (\Phi_1^{V_M})^{-1}(U_{m_k}))_k$) that converges to a point $p$. Since convergence to a point is preserved under homeomorphisms, we conclude that $(U_{m_k})_k$ converges to the point $\Phi_1^{V_M}(p)$. Thus we see that $(K, \mathcal{U})$ is a generalized cell decomposition of $D^n$.

\begin{remark}
    When the refinement of one $G$-complex by another is introduced in \cite{Markovic13},  the fact that the refinement still satisfies property $(\star)$ is not addressed. 
    Since we are unable to prove directly that the refinement satisfies $(\star)$, we use an equivalent topological formulation of $(\star)$ that we can work with: the third property in our definition of generalized cell decompositions.
\end{remark}

We are left to define the $G$-action on $(K, \mathcal{U})$ and show that it satisfies all the properties of a $G$-complex. Let $g \in G$. On the subset $K_1 \subset K$, we set $\mu(g) = \mu_1(g)$. The remainder of $K$ can be written as the union $\coprod_{V \in \mathcal{U}_1} K \cap V$ and we define $\mu(g)$ on each of these sets individually. Let $V \in \mathcal{U}$. On $K \cap \overline{V}$, we define
\[ \mu(g) := \Phi_1^{\mu_1(g)(V)} \circ \mu_2(g) \circ \left( \Phi_1^{V} \right)^{-1}. \]
We need to show that $\mu(g)$ is well-defined and a homeomorphism. Indeed, on the set $K \cap \partial V = \partial V \subset K_1$, we have overlapping definitions of $\mu(g)$ and we need to check the following equality.
\[ \forall x \in \partial V : \mu_1(g) = \Phi_1^{\mu_1(g)(V)} \circ \mu_2(g) \circ \left( \Phi_1^{V} \right)^{-1}. \]
This is equivalent to the equation
\[ \forall x \in S^{n-1} : \mu_1(g) \circ \Phi_1^{V} = \Phi_1^{\mu_1(g)(V)} \circ \mu_2(g). \]
By assumption, $\mu_1(g)$ and $\mu_2(g)$ coincide on $S^{n-1}$ and the marking $\Phi_1^{\mathcal{U}_1}$ commutes with the action $\mu_1$ on $S^{n-1}$. Therefore, this equality holds and we conclude that $\mu(g)$ is well-defined.

Next, we show that $\mu(g)$ is a homeomorphism on $K$. By construction, $\mu(g)$ is defined as a continuous map on various closed subsets of $K$ that cover $K$. Thus, $\mu(g)$ is continuous. Furthermore, we know that $\mu(g)$ sends $K_1$ to itself and $K \cap V$ bijectively to $K \cap \mu_1(g)(V)$. Since all $K \cap V$ are pairwise disjoint and $\mu_1(g)$ acts bijectively on $\mathcal{U}_1$, we see that $\mu(g)$ is bijective. Since $K$ is compact and Hausdorff, $\mu(g)$ is a homeomorphism.

We are left to show that the map $\mu : G \rightarrow \Homeo(K)$ is a homomorphism. Recall that we are using the convention that actions by homeomorphisms satisfy the formula $\mu(gh) = \mu(h) \circ \mu(g)$. Let $g, h \in G$. On $K_1$, we simply have
\[ \mu(gh) = \mu_1(gh) = \mu_1(h) \circ \mu_1(g) = \mu(h) \circ \mu(g) \]
For $V \in \mathcal{U}_1$, we compute
\begin{equation*}
    \begin{split}
        \mu(gh) & = \Phi_1^{\mu_1(gh)(V)} \circ \mu_2(gh) \circ \left( \Phi_1^{V} \right)\\
        & = \Phi_1^{\mu_1(h)(\mu_1(g)(V))} \circ \mu_2(h) \circ \mu_2(g) \circ \left( \Phi_1^{V} \right)^{-1}\\
        & = \Phi_1^{\mu_1(h)(\mu_1(g)(V))} \circ \mu_2(h) \circ \left( \Phi_1^{\mu_1(g)(V)} \right)^{-1} \circ \Phi_1^{\mu_1(g)(V)} \circ \mu_2(g) \circ \left( \Phi_1^{V} \right)^{-1}\\
        & = \mu(h) \circ \mu(g).
    \end{split}
\end{equation*}
We conclude that $\mu$ is a $G$-action, turning $(\mu, K, \mathcal{U}, \Phi^{\mathcal{U}})$ into a marked $G$-complex. One immediately sees that $(\mu, K, \mathcal{U})$ is a refinement of $(\mu_1, K_1, \mathcal{U}_1)$.\\

\begin{remark}
    Let $(\mu_i, K_i, \mathcal{U}_i, \Phi_i^{\mathcal{U}_i})$ for $i = 1, 2, 3$ be three marked $G$-complexes. For short, we simply denote them $\mathcal{K}_1, \mathcal{K}_2, \mathcal{K}_3$. Let $\mathcal{K}_{i, j}$ denote the refinement of $\mathcal{K}_i$ induced by $\mathcal{K}_j$. It is an easy exercise to show that the refinement of $\mathcal{K}_{1,2}$ induced by $\mathcal{K}_3$ is the same as the refinement of $\mathcal{K}_1$ induced by $\mathcal{K}_{2,3}$. In other words, the refinement of $G$-complexes is associative. We never use this, but think it is worth pointing out.
\end{remark}

The following proposition summarizes and generalizes Propositions 3.2, 3.3, and 3.4 from \cite{Markovic13} to the extent that we need them. Their proofs are identical to the $D^3$-case, short, and do not use the topology or geometry of $D^{n}$ at all, which is why we give minimal comment and refer to the proofs given by Markovic.

\begin{proposition} \label{prop:Markovic}
    Let $(\mu_i, K_i, \mathcal{U}_i)$ for $i \in \{ 1, 2 \}$ be two $G$-complexes on $D^n$ and let $(\mu, K, \mathcal{U})$ be the refinement of $(\mu_1, K_1, \mathcal{U}_1)$ induced by $(\mu_2, K_2, \mathcal{U}_2)$. Then the following three statements hold:
        \begin{enumerate}
        \item If $(\mu_1, K_1, \mathcal{U}_1)$ 
        and  $(\mu_2, K_2, \mathcal{U}_2)$ are convergence $G$-complexes, then $(\mu, K, \mathcal{U})$ is a convergence $G$-complex.

        \item If $(\mu_1, K_1, \mathcal{U}_1)$ and $(\mu_2, K_2, \mathcal{U}_2)$ are free $G$-complexes, then $(\mu, K, \mathcal{U})$ is a free $G$-complex.

        \item For every $U \in \mathcal{U}$, there exist $U_1 \in \mathcal{U}_1$, $U_2 \in \mathcal{U}_2$ such that
        \[ \Stab_{\mu}(U) = \Stab_{\mu_1}(U_1) \cap \Stab_{\mu_2}(U_2). \]
    \end{enumerate}


\end{proposition}

We need one final proposition, which generalizes to $D^{n}$ due to the fact that radial extensions exist in higher dimensions as well.

\begin{proposition} \label{prop:radial.extension}
    Let $(\mu, K, \mathcal{U})$ be a free convergence $G$-complex on $D^n$ 
    such that $\Stab_{\mu}(U)$ is trivial for every $U \in \mathcal{U}$. Let $K_0 := D^{n}$. Then there exists a free convergence $G$-complex $(\mu_0, K_0, \emptyset)$ that is a refinement of the $G$-complex $(\mu, K, \mathcal{U})$.

    That is, the action of the convergence group $\mu(G)$ on $S^{n-1}$ can be extended to a free convergence action on $D^{n}$.
\end{proposition}

\begin{proof}
    The $G$-complex $(\mu_0, K_0, \emptyset)$ will be the refinement of $(\mu, K, \mathcal{U})$ induced by the following $G$-complex:

    For every $g \in G$, we define $\mu_{rad}(g) : D^{n} \rightarrow D^{n}$ to be the radial extension of the homeomorphism $\mu(g) : S^{n-1} \rightarrow S^{n-1}$. Since the radial extension defines a monomorphism $\Homeo(S^{n-1}) \rightarrow \Homeo(D^{n})$, we obtain a $G$-complex $(\mu_{rad}, K_0, \emptyset)$.

    Set $(\mu_0, K_0, \emptyset)$ to be the refinement of $(\mu, K, \mathcal{U})$ induced by $(\mu_{rad}, K_0, \emptyset)$. By assumption, $\Stab_{\mu}(U)$ is trivial for every $U \in \mathcal{U}$. Therefore, it follows from Proposition \ref{prop:Markovic} (2) that $(\mu_0, K_0, \emptyset)$ is a free $G$-complex and from Proposition \ref{prop:Markovic} (1) that it is a convergence $G$-complex.
\end{proof}




\section{Uniformizing Sierpinski Spaces}\label{sec:Uniformizing-Sierpinski}
In this section we prove a result concerning embeddings of the Sierpinski space $\mathcal{M}_{n-1}^n$ into $S^n$.
We recall the definition of $\mathcal{M}_{n-1}^n$ and some of its properties. Let $I=[0,1]$ be the unit interval, let $\C\subset I$ denote the middle-thirds Cantor set, and let $\chi\colon I\rightarrow\{0,1\}$ be the indicator function for $\C$. 

\begin{definition} Let $I^n$ be the $n$-cube with coordinates $\mathbf{x}=(x_1,\ldots, x_n)$. For each $n\geq 1$ and each $k$ satisfying  $0\leq k\leq n$, we define a space $\mathcal{M}^n_k$ as follows.  For each $(n-k)$-element subset $\alpha\subseteq \{1,\ldots, n\}$, consider the function $f_\alpha(\mathbf{x})=\prod_{i\in \alpha}\chi(x_{i})$, and let $F_k(\mathbf{x})=\sum_{\alpha}f_\alpha(\mathbf{x})$. By convention, we define $F_n(\mathbf{x})\equiv1$. If $V_k=\{\mathbf{x}\in I^n\mid F_k(\mathbf{x})=0\}$ then $\mathcal{M}_k^n=I^n\setminus V_k$.
\end{definition}

By construction, $\mathcal{M}^n_k$ is the set of points with at least $n-k$ coordinates in $\C$. In particular, $\mathcal{M}^n_0$ is the set of points for which every coordinate lies in  $\C$, \emph{i.e.} $\C^n\subset I^n$, while $\mathcal{M}_{n}^n$ is $I^n$. The topological dimension of $\mathcal{M}_k^n$ is $k$, and it satisfies a universal property: any compact $k$-dimensional metric space which embeds into $\R^n$, embeds in $\calM^n_k$. In this sense, $\calM^n_k$ is the universal receptor for $k$-dimensional compacta in $\R^n$. 

By embedding $I^n$ into $S^n$, we can regard $\calM^n_k$ as a compact subset of $S^n$.  When $k=n-1$, each complementary component of $\calM^n_{n-1}$ is an open disk with boundary homeomorphic to $S^{n-1}(\cong \partial I^{n}$). These spheres are called peripheral.  With respect to the standard metric on $S^n$, the collection of peripheral spheres form a null collection, defined as follows. 

 \begin{definition}
        A collection of $\{\sigma_k\}_{k=\infty}^\infty$ of pairwise disjoint embedded $(n-1)$-spheres in $S^n$ (with its standard metric) is called a \emph{null collection} if $\lim_{k\to\infty}\diam(\sigma_k)=0.$
    \end{definition}

In the case of $\calM^n_{n-1}$, if one considers the complement of all the peripheral spheres, then the closure of each component is either  $I^n$ or $\calM^n_{n-1}$. We will show below that for a dense null collection of locally flat spheres, the closure of each component of the complement is homeomorphic to $\calM^n_{n-1}$, and that there are at most countably many such components.  For this, we will need the following definition.

    \begin{definition}
        An $n$-triod $\Upsilon^n$ is the topological space formed as the union of the $n$-disk $D^n$ and an interval $I$ such that $I\cap D^n$ is one of the endpoints of $I$ and an interior point of $D^n$.
    \end{definition}

    Since an $\Upsilon^{n-1}$ is $(n-1)$-dimensional, and clearly embeds in $\R^n$, it also embeds in $\calM^n_{n-1}$. In fact, there exists an embedding of $\Upsilon^{n-1}$ into the ``interior" of $\calM^n_{n-1}$.
    \begin{lemma}\label{lem:Disjoint-Triod}
        There exists an embedding of $\Upsilon^{n-1}$ into $\calM^{n}_{n-1}$ whose image is disjoint from every peripheral sphere.
    \end{lemma}
    \begin{proof}
        Recall that $\calM^n_{n-1}$ is the set of points in $I^n$ where at least one coordinate is in $\C$. Let $x_1$ be a point of $\C$ that is not a boundary point (i.e. not in the closure of one of the complementary intervals of $\C$). Then $\{x_1\}\times \interior(I^{n-1})\subset \calM^n_{n-1}$, and is disjoint from every peripheral sphere. We can then find a closed $(n-1)$-disk $D$ contained in $\{x_1\}\times \interior(I^{n-1})$. Now choose $x_2,\ldots,x_n \in \mathcal{C}$ that are not boundary points such that in $p=(x_1,\ldots x_n)\in \interior(D)$. Out of $p$, adjoin a small arc $\gamma=\{(x_1 + t,\ldots, x_{n-1}, x_n)\mid 0 \leq t\leq \epsilon\}$ for some $\epsilon<1 - x_1$. 
        Then $\gamma$ is also disjoint from every peripheral sphere, and meets $D$ only at $p$. Thus $\Upsilon=D\cup \gamma$ is the desired $(n-1)$-triod.  
    \end{proof}

    \begin{lemma}\label{lem:Countable-Spongy}
    Let $\{\Sigma_k\}_{k=1}^\infty$ be a dense null sequence of locally flat (n-1)-spheres embedded in $S^n$. Then $S^n\setminus ( \cup_{k=1}^\infty\Sigma_k)$ has countably many connected components, and the closure of each is homeomorphic to $\calM^n_{n-1}$.
    \end{lemma}
    \begin{proof}
    We first prove the second claim. Let $\calK$ be the closure of a component of $S^n\setminus (\cup_{k=1}^\infty\Sigma_k)$. Since the $\Sigma_k$ are dense and locally flat in $S^n$, the closure $\calK$ is nowhere dense. Moreover, because each $\Sigma_k$ is locally flat, $S^n\setminus \calK$ consists of a union of tame n-balls $\{D_i\}$. The boundary of each $D_i$ is one of the $\Sigma_k$, hence $\overline{D}_i\cap \overline{D}_j=\emptyset$ if $i\neq j$. The fact that the $\Sigma_k$ form a null collection implies that the $\overline{D}_i$ do too. By \cite[Theorem 1]{Cannon-Sierpinski}, $\calK$ is homeomorphic to $\calM^n_{n-1}$. (The case $n=4$ follows from {\cite[Theorem 1.3]{BanakhRepovs}}.)

    Now we prove that there are only countably many components of $S^n\setminus (\cup_{k=1}^\infty\Sigma_k)$.  Let $\calK_i$, $i\in I$ be the closures of components of $S^n \setminus (\cup_{k=1}^\infty\Sigma_k)$.  Observe that the intersection $\calK_i\cap \calK_j$ is either empty  or equal to $\Sigma_k$ for some $k$. By Lemma \ref{lem:Disjoint-Triod}, in each $\calK_i$ we can find an embedded $(n-1)$-triod $\Upsilon_i\cong\Upsilon^{n-1}\subset\calK_i$ which does not meet any  peripheral sphere of $\calK_i$. Since $\calK_i$ and $\calK_j$ meet only in a peripheral sphere if at all
    , the triods will be disjoint. Young \cite{Young} generalized Moore's triod theorem to all dimensions, proving  that $\mathbb{R}^n$ (and thus $S^n$) contains at most countably many pairwise disjoint embedded $(n-1)$-triods. 
    Hence there can be at most countably many $\calK_i$. This proves the first claim, and the lemma.
    \end{proof}

    \begin{proposition}\label{prop:Uniformize-Sierpinski} Let $\{\Sigma_k\}_{k=1}^\infty$ be a dense null sequence of locally flat $(n-1)$-spheres embedded in $S^n$. Then there exists a homeomorphism $F\colon S^n\rightarrow S^n$ such that $F(\Sigma_k)$ is round for each $k$.
    \end{proposition}

    \begin{proof}
       By Lemma \ref{lem:Countable-Spongy}, there are countably many components of $S^n\setminus (\cup_{k=1}^\infty\Sigma_k)$, and the closure of each is homeomorphic to $\calM_{n-1}^n.$ Enumerate the closures of these complementary components as $\{\calK_j\}_{j=1}^\infty$.  Clearly we have $\cup_{j=1}^\infty\calK_j=S^n$ and as we observed in the proof of Lemma \ref{lem:Countable-Spongy}, for $i\neq j$, the intersection $\calK_i\cap \calK_j$ is empty or some $\Sigma_k$. Conversely, each $\Sigma_k$ is the intersection of exactly two of the $\calK_i$. 
       
       Now let $\Delta_r$ be a dense null collection of round $(n-1)$-spheres in $S^n$.  The closures of $S^n\setminus (\cup_{m=1}^\infty \Delta_r)$ are also homeomorphic to $\calM_{n-1}^n$, and again there are countably many of them, which we denote by $\calL_m,$ $m\geq 1$.  We build $F\colon S^n\rightarrow S^n$ as follows. First set $j_1=m_1=1$ and choose a homeomorphism $F_1\colon\calK_1\rightarrow \calL_1$. Now suppose we have defined $F_N : \bigcup_{i = 1}^N \mathcal{K}_{j_i} \rightarrow \bigcup_{i=1}^N \mathcal{L}_{j_i}$. Choose some peripheral sphere of $\Sigma_{j_{N+1}}\subset\cup_{i=1}^N\calK_{j_i}$ and let $\Delta_{m_{N+1}}\subset \cup_{i=1}^N\calL_{j_i}$ be the corresponding peripheral sphere. There is a unique $\calK_{j_{N+1}}$ not among $\calK_{j_1},\ldots, \calK_{j_N}$ meeting $\Sigma_{j_{N+1}}$. Let $\calL_{m_{N+1}}$ be the corresponding element meeting $\Delta_{m_{N+1}}$.  By \cite[Lemma 1]{Cannon-Sierpinski}, there exists a homeomorphism $h_{N+1}\colon \calK_{j_{N+1}}\rightarrow \calL_{j_{N+1}}$ extending $F_N\colon \Sigma_{j_{N+1}}\rightarrow \Delta_{m_{N+1}}$. (The proof Cannon gives also works when $n=4$ case, once one knows that the annulus theorem holds. This was proven by Quinn \cite{Quinn-Ends-III,Edwards-Annulus}.)  We use this to define $F_{N+1}\colon \cup_{i=1}^{N+1}\calK_{j_i}\rightarrow \cup_{i=1}^{N+1}\calL_{m_i}$, by pasting $F_N$ to $h_{N+1}$ along $\Sigma_{j_{N+1}}$, where they agree. 
       
       Since $F_N$ agrees with $F_{N+1}$ on $\cup_{i=1}^{N+1}\calK_{j_i}$, in the limit we get a well-defined map $F\colon S^n\rightarrow S^n$. Injectivity is clear from the definition, hence $F$ is an embedding as $S^n$ is compact and Hausdorff.  In particular, the image of $F$ is closed. On the other hand, the domain and codomain have the same dimension so $F$ is also open. Since $S^n$ is connected, $F$ must be onto, hence a homeomorphism. 
    \end{proof}

\section{Proof of the main theorem}\label{sec:General-Markovic}
In this section we prove Theorem \ref{thm:Sn-cubulated}.  Let us first recall the hypotheses. We assume that $G$ is a torsion-free hyperbolic group whose boundary $\dinf G$ is homeomorphic to $S^{n-1}$, $n\neq 5$. Moreover $G$ contains a finite collection of quasi-convex, codimension-1 subgroups $\mathcal{H}$ satisfying:
\begin{enumerate}
    \item $\mathcal{H}$ separates pairs of points at infinity.
    \item For each $H\in \mathcal{H}$, $\dinf H$ is a locally flat $S^{n-2}$ in $S^{n-1}\cong \dinf G.$
\end{enumerate}
By (1), we know that $\mathcal{H}$ cubulates $G$; that is, $G$ acts geometrically on a finite-dimensional CAT(0) cube complex whose hyperplane stabilizers are conjugates of the $H_i\in \mathcal{H}$.

Recall that $G$ acts on its boundary by homeomorphisms, inducing an action by automorphisms on the cohomology $H^{n-1}(\dinf G)\cong H^{n-1}(S^{n-1})\cong \Z$. By passing to a subgroup of index at most $2$, we may therefore assume that the action of $G$ on $\dinf G$ is orientation-preserving. We denote the action of $G$ on the boundary $S^{n-1}$ by a representation $\rho\colon G\rightarrow \Homeo^+(S^{n-1})$. 

\begin{lemma}\label{lem:Orientable-Manifold-Subgroups}
    After replacing $\mathcal{H}$ by a family of finite index subgroups $\mathcal{H}'$, we may assume that  each $H\in \mathcal{H}'$ additionally satisfies:
    \begin{enumerate}
        \item[(3)] There exists a closed, orientable $n-1$-manifold $N_H$ such that $H\cong \pi_1(N_H)$. Additionally, $H$ preserves each complementary component of $\partial_{\infty} G \setminus \partial_{\infty} H$ and acts freely, properly discontinuously cocompactly on each.
    \end{enumerate}
\end{lemma}
\begin{proof}
    Let $H\in \mathcal{H}$. By (2) and the generalized Schoenflies theorem, the pair $(\dinf G, \dinf H)$  is homeomorphic to the standard pair $(S^{n-1}, S^{n-2})$.  In particular, $\dinf G\setminus \dinf H$ consists of exactly two components, each homeomorphic to $\R^{n-1}$ and each of whose closures is homeomorphic to the closed $n-1$-ball $\D^{n-1}$. 
    
    Since the action of $H$ on $\dinf G$ preserves $\dinf H$ it also preserves the complement.  Therefore, there exists a subgroup $H'\leq H$ of index at most $2$ such that $H'$ preserves each complementary component of $\dinf G\setminus \dinf H$. Because $H'$ has finite index in $H$, the limit set of $H'$ is the same as that of $H$, \emph{i.e.} $\dinf H'=\dinf H$. Since the action of $G$ on $S^{n-1}$ is orientation-preserving, the action of $H'$ on $S^{n-2}$ must also be orientation-preserving. 
    
    Let $U$ be one of the two components of $\dinf G\setminus \dinf H'$. By a result of Swenson \cite{Swenson}, $H'$ acts freely, properly discontinuously and cocompactly on $U$.  Since $U\cong \R^{n-1}$, this implies that $H'$ is isomorphic to the fundamental group of  $U/H'=N_{H'}$, which is a closed, orientable, aspherical $n-1$-manifold. Now the collection $\mathcal{H}'$ obtained by replacing each $H\in \mathcal{H}$ with its corresponding subgroup $H'$ satisfies (1), (2) and (3). 
    %
    %
\end{proof}


We now assume $\mathcal{H}$ satisfies conditions (1)--(3).  Applying Theorem \ref{thm:Markovic-Malnormal}, we may assume that $\mathcal{H}$ also satisfies
\begin{enumerate}
    \item[(4)] Every $H\in \mathcal{H}$ is malnormal.
\end{enumerate}


Let $\Lambda_i$ denote the limit set of $H_i$, which can be identified with $\dinf H_i$.  The action of each $H_i$ preserves the limit set $\Lambda_i$, which by assumption is homeomorphic to a locally flat $(n-2)$-sphere.  
Now consider $H_i^G :=\{H_i^g=gH_ig^{-1}\mid g\in G\}$, the set of conjugates of $H_i$.  The limit set of $H_i^g$ is $g\Lambda_i$, which is disjoint from $\Lambda_i$ by malnormality. Each translate $g\Lambda_i$ separates $S^{n-1}$ into two components, hence we may define the dual tree $T_i$ whose edges are translates $g\Lambda_i$ and whose vertices are in 1-1 correspondence with connected components of \[S^{n-1}\setminus \left(\bigcup_{g\in G}g\Lambda_i\right).\]

For any $H=H_i^g\in H_i^G$, we use the orientation on $S^{n-1}$ to choose a left side $L_H$ and a right side $R_H$ of $\Lambda_H=g\Lambda_i$. The sides $L_H$ and $R_H$ correspond to the two half-edges at the midpoint of each edge of $T_i$. By $(3)$, $G$ acts on $T_i$ without inversions, hence for any $g'\in G$, the left side (resp. right side) of $\Lambda_H$  is taken to the left side (resp. right side) of $g'\Lambda_H=\Lambda_{H^{g'}}$.

\begin{lemma}\label{lem:Left-to-Right-Homeo}For $H\in H_i^G$, there exists a collection of homeomorphisms $\sigma_H\colon \overline{L}_H\rightarrow \overline{R}_H$ satisfying
\begin{enumerate}[(i)]
\item $\sigma_H$ restricts to the identity on $\Lambda_H$.
\item For any $g\in G$, $H\in H_i^G$ and  $x\in L_H$ we have \[(\sigma_{H^g}\circ \rho(g))(x)=(\rho(g)\circ \sigma_H)(x).\]
\end{enumerate}
    
\end{lemma}
\begin{proof}
     By (3), $H_i$ acts on both sides $L_{H_i},~R_{H_i}$ of $S^{n-1}\setminus \Lambda_i$ properly discontinuously, cocompactly. Define $N_i^L=L_{H_i}/H_i$, and $N_i^R=R_{H_i}/H_i$. Choosing basepoints in $N_i^L$ and $N_i^R$, we get an identification between $\pi_1(N_i^L)$ and $\pi_1(N_i^R)$. Since both of these are aspherical, we obtain a homotopy equivalence $h_i\colon N_i^L\rightarrow N_i^R$. Our choice of $h_i$ yields a commutative diagram \[\xymatrix{\pi_1(N_i^L)\ar[dr]\ar[rr]^{(h_i)_*}&&\pi_1(N_i^R)\ar[dl]\\& H_i &}\]
     and in particular the lift $\wt{h}_i
     \colon \wt{N}_i^L\rightarrow \wt{N}_i^R$ to universal covers induces the identity map on the boundary.  By Corollary \ref{cor:Hyperbolic-Borel-Except-4}, $h_i$ is homotopic to a homeomorphism $f_i\colon N_i^L\rightarrow N_i^R.$ (This is where we need that $n \neq 5$.)  Since $f_i$ is homotopic to $h_i$, the same holds for the lift  $\wt{f}_i$. Set $\sigma_{H_i}=\wt{f}_i\cup \id_{\Lambda_{H_i}} \colon \overline{L}_{H_i}\rightarrow \overline{R}_{H_i}$; by construction $\sigma_{H_i}\circ\rho(h)(x)=\rho(h)\circ \sigma_{H_i}(x)$ for every $h\in H_i$ and $x\in \overl{L}_{H_i}$.
     
     Now suppose $H=H_i^{g_0}$ and define $\sigma_H=\rho(g_0)\circ \sigma_{H_i}\circ\rho(g_0)^{-1}$. Since the action of $g$ takes the left side of $H_i$ to the left side of $H$, $\sigma_H\colon \overline{L}_{H}\rightarrow \overline{R}_{H}$ is a homeomorphism and $\sigma_H|_{\Lambda_H}$  is the identity since the same holds for $\sigma_{H_i}$, which proves (i). Let $g_1\in G$ be another group element satisfying $H_i^{g_1}=H$, whence $g_1=g_0h$ for some $h\in H_i$. If we define $\sigma_H'=\rho(g_1)\sigma_{H_i}\rho(g_1)^{-1}$ then for any $x\in \overl{L}_H$
     \begin{eqnarray*}\sigma_H'(x)&=&\rho(g_1)\circ\sigma_{H_i}\circ\rho(g_1)^{-1}(x)\\
     &=&\rho(g_0)\circ\rho(h)\circ\sigma_{H_i}\circ\rho(h)^{-1}\circ\rho(g_0)^{-1}(x)\\&=&\rho(g_0)\circ\sigma_{H_i}\circ\rho(g_0)^{-1}(x)\\&=&\sigma_H(x)
     \end{eqnarray*}
     where the third equality follows from $H_i$-equivariance of $\sigma_{H_i}$. Thus $\sigma_H$ depends only the coset $gH_i$. The equality in (ii) now follows easily.
\end{proof}


Combined with results of \S \ref{sec:Uniformizing-Sierpinski}, we now use the homeomorphisms constructed above to construct a generalized cell decomposition of $D^n$ for each $H_i\in \mathcal{H}.$

\begin{lemma}\label{lem:QC-Gives-Cell-Structure}Let $\D^n$ denote the closed $n$-disk with the Euclidean metric, with boundary $\Sa^{n-1}$. There exists a homeomorphism $f_i\colon D^n\rightarrow \D^n$ and a collection of pairwise disjoint, properly embedded $(n-1)$-disks $B_{H}\cong \D^{n-1}\subset \D^n$, $H\in H_i^G$ satisfying the following properties:
\begin{enumerate}[(i)]
    \item $\partial B_H=f_i(\Lambda_H)$.
    \item If $\diam(f_i(\Lambda_H))\to 0$ then $\diam(B_H)\to 0$.
    \item Let \[K_i=\Sa^{n-1}\cup \left(\bigcup_{H\in H_i^G}B_H\right),\] and let $\mathcal{U}_i$ be the set of complementary components of $\D^{n}\setminus K_i$.  Then $K_i$ is closed and each pair $(U,\partial U)$, $U\in \mathcal{U}_i$ is homeomorphic to $(D^n,S^{n-1})$.
    \item Let $U\in \mathcal{U}_i$ and $a,b\in \partial U\cap S^{n-1}$. Then for each $H\in H_i^g$, $f_i(\Lambda_H)$ does not separate $a$ and $b$.
\end{enumerate}
\begin{proof}
    For each $\epsilon>0$, there are only finitely many $\Lambda_H,$ $H\in H_i^G$ such that $\diam(\Lambda_H)\geq \epsilon.$ The union $\cup_{H\in H^{G}_i}\Lambda_H$ is dense in $S^n$ because it contains the orbit of the endpoints of one axis of an element of $G$. Therefore $\{\Lambda_H\mid H\in H_i^G\}$ forms a dense null collection. By Proposition \ref{prop:Uniformize-Sierpinski}, there exists a homeomorphism $f_i\colon S^{n-1}\rightarrow \Sa^{n-1}$ such that $f_i(\Lambda_H)$ is a round $(n-2)$-sphere, for each $H\in H_i^G$. By the Alexander trick, $f_i$ extends to a homeomorphism $f_i\colon D^n\rightarrow \D^n$.

    Regarding $\interior(\D^{n})$ as hyperbolic $n$-space in the Poincar\'e model, for each $H\in H_i^G$, $f_i(\Lambda_H)$ bounds a round hemisphere $B_H$, where $B_H\cap B_{H'}$ are disjoint if $H\neq H'$. In the Euclidean metric, $\diam(B_H)=\diam(f_i(\Lambda_{B_H}))$, proving (i) and (ii). The remainder of the proof is exactly the same as that of \cite[Lemma 3.2]{Markovic13}.
\end{proof}
    
\end{lemma}

Thus the pair $(K_i, \mathcal{U}_i)$ is a generalized cell decomposition. The next result shows that each $H_i\in \mathcal{H}$ yields a free convergence $G$-complex, and when combined with the results of Section \ref{sec:G-complexes}, it will be the basis for the induction in the proof of the main theorem.

\begin{lemma}\label{lem:QC-Gives-Free-Convergence}There is a free convergence $G$-complex $(\rho_i, K_i, \mathcal{U}_i)$ satisfying:
\begin{enumerate}
    \item[(5)]For any $U\in \mathcal{U}_i$, any $g\neq 1\in \Stab_{\rho_i}(U)$ and any $H\in H^G_i$, $\Lambda_H$ does not separate the fixed points of $g$ in $\partial U$.
\end{enumerate}
\end{lemma}
\begin{proof}[Proof of Lemma \ref{lem:QC-Gives-Free-Convergence}] Let ($K_i,\mathcal{U}_i)$ be the generalized cell decomposition from Lemma \ref{lem:QC-Gives-Cell-Structure}. For each $H\in H^G_i$, choose a homeomorphism $\psi_H\colon \overl{L}_H\rightarrow \overl{B}_H$ which restricts to the identity on $\Lambda_H$. Let $\rho_0$ be the action of $G$ on $S^{n-1}$. We extend this to an action $\rho_i$ on $K_i$ via the formula:
\begin{equation}\label{eq:Equivariant-BH}\tag{$\dagger$}
    \rho_i(g)(x)=(\psi_{gHg^{-1}}\circ \rho_0(g)\circ\psi_H^{-1})(x), \textrm{ for } x\in B_H.
\end{equation}
 This is a homeomorphism of $K_i$ since it agrees with $\rho_0$ on $S^{n-1}$ and the interiors $\interior(B_H)$ are pairwise disjoint. A simple calculation using Equation (\ref{eq:Equivariant-BH}) verifies that $\rho_i$ defines a $G$-action. The proof that $(\rho_i,\mathcal{U}_i,K_i)$ is a free, convergence $G$-action follows exactly as in  \cite[Lemma 3.2]{Markovic13}, substituting Lemma \ref{lem:Left-to-Right-Homeo} for \cite[Proposition 3.6]{Markovic13} and Lemma \ref{lem:QC-Gives-Cell-Structure} for \cite[Proposition 3.7]{Markovic13}.
\end{proof}

We now combine the results of this section to deduce Theorem \ref{thm:Sn-cubulated}.  

\begin{proof}[Proof of \ref{thm:Sn-cubulated}] Let $G$ be a torsion-free hyperbolic group such that $\dinf G\cong S^{n-1}$, and let $\mathcal{H}=\{H_1,\ldots, H_k\}$ be quasi-convex, codimension-1 subgroups which satisfy (1) and (2). By Lemma \ref{lem:Orientable-Manifold-Subgroups} and Theorem \ref{thm:Markovic-Malnormal}, we may assume each $H_i\in\mathcal{H}$ satisfies (1)--(4). Now by Lemmas \ref{lem:QC-Gives-Cell-Structure} and \ref{lem:QC-Gives-Free-Convergence}, for $1\leq i\leq k$, there exists  a homeomorphism $f_i\colon D^n\rightarrow \D^n$ and a free convergence $G$-complex $(\rho_i,K_i,\mathcal{U}_i)$ satisfying (5), such that $K_i$ is a union of $\Sa^{n-1}$ with a disjoint collection of round hemispheres $\B_H$, $H\in H_i^G$.

For $1\leq i\leq k$, we now define a sequence of free, convergence $G$-complexes $(\mu_i,L_i,\mathcal{V}_i)$ that is a refinement of $(\rho_i,K_i,\mathcal{U}_i)$. To start, set $(\mu_1,L_1,\mathcal{V}_1)=(\rho_1,K_1,\mathcal{U}_1)$, and then inductively define $(\mu_{i+1},L_{i+1},\mathcal{V}_{i+1})$ as the refinement of $(\rho_{i+1},K_{i+1},\mathcal{U}_{i+1})$ by the pushforward $(F_{i+1})_*(\mu_i,L_i,\mathcal{V}_i)$, where $F_{i+1}:=f_{i+1}\circ f_i^{-1}$. By construction, $(\rho_{i+1}, L_{i+1}, \mathcal{V}_{i+1})$ is a $G$-complex; moreover, it is a free, convergence $G$-complex by Proposition \ref{prop:Markovic} (1) and (2).


Now consider $(\mu_k,\mathcal{V}_k,L_k)$. By Proposition \ref{prop:Markovic} (3), for any $V\in \mathcal{V}_k$ there exist $U_i\in \mathcal{U}_i$ such that $\Stab_{\mu_k}(V)=\cap_{i=1}^k\Stab_{\rho_{i}}(U_i)$. Suppose that there exists $g\neq 1\in \Stab_{\mu_k}(V)$, and let $\xi^{\pm}$ be the fixed points of $g\in S^{n-1}.$ Since $g\in \Stab_{\rho_i}(U_i)$, then by (5) we know that $\xi^+$ and $\xi^-$ are not separated by $\Lambda_H$ for any $H\in H_i^G$ for each $i$.  But this holds for every $H_i\in \mathcal{H}$, contradicting our assumption that $\mathcal{H}$ separates pairs of points in $\dinf G$. Therefore $\Stab_{\mu_k}(V)$ is trivial for every $V\in \mathcal{V}_k$. The proof now follows from  Propositions \ref{prop:radial.extension} and \ref{prop:quotient.of.maximal.G-action}.
\end{proof}




\section{Characterization of local flatness for codimension-1 submanifolds} \label{sec:local.flatness.of.codimension.one.submanifolds}

\subsection{From simple connectedness to local flatness} \label{sec:simply.connected.implies.standard}

Let $G$ be a torsion-free, hyperbolic group such that $\partial_{\infty} G \approx S^{n-1}$. Let $H < G$ be a quasi-convex, codimension-1 subgroup such that its limit set $\Lambda \subset \partial_{\infty} G$ is homeomorphic to $S^{n-2}$. By the Jordan-Brouwer separation theorem, $\partial_{\infty} G \setminus \Lambda$ has exactly two connected components, which we denote by $Z^{+}$ and $Z^{-}$. 


Our goal in this section is to prove that if $Z^{+}$ and $Z^{-}$ are simply connected, then the embedding $\Lambda \hookrightarrow \partial_{\infty} G$ is equivalent to the standard embedding $S^{n-2} \hookrightarrow S^{n-1}$. To show this, we need a series of definitions from topology. We do not give much context to these definitions, as they merely appear as assumptions in established theorems that we need to invoke.

\begin{definition} \label{def:ANR}
An {\it absolute neighborhood retract} (ANR) is a normal topological space $X$ such that for every normal space $Z$, every closed subset $Y \subset Z$ and every continuous map $f : Y \rightarrow X$ there exists an open neighborhood $U$ of $Y$ such that $f$ has a continuous extension $U \rightarrow X$.
\end{definition}

We will need two important facts about ANRs which are summarised in the following Lemma.

\begin{lemma}[{\cite[Theorem 5 \& Theorem 14]{Palais66}}] \label{lem:facts.about.ANR}
    Paracompact Hausdorff spaces which are locally ANRs are ANRs. (This includes open topological manifolds.) Furthermore, any ANR has the homotopy type of a CW-complex.
\end{lemma}



\begin{definition} \label{def:kLLC}
	Let $Y$ be a compact space of dimension $k$ and $Z \subset Y$ a closed subspace. We say that $Z$ is {\it $k$-LCC} in $Y$, if for every $z \in Z$ and every neighborhood $U$ of $z$, there exists a neighborhood $V$ of $z$ such that every continuous map $\alpha : S^k \rightarrow V \setminus Z$ extends continuously to $\tilde{\alpha} : D^{k+1} \rightarrow U \setminus Z$.
\end{definition}


Let $Z$ denote either $Z^{+}$ or $Z^{-}$. Observe that $Z$ is an open subspace of the compact manifold $S^{n-1}$ 
and thus an ANR. In particular, $Z$ has the homotopy type of a CW complex. We begin with the following result.

\begin{lemma} \label{lem:simply.connected.implies.contractibility}
    If $Z$ is simply connected, then $Z$ is contracible.
\end{lemma}

\begin{proof}
        As $\Lambda \cong S^{n-2}$ is locally contractible, Alexander duality implies that $\tilde{H}_i(S^{n-1} \setminus \Lambda) \cong \tilde{H}^{n-2-i}(\Lambda)$. Since $S^{n-1} \setminus \Lambda = Z^{+} \coprod Z^{-}$, we know that the homology of $Z$ appears as a summand in $\tilde{H}_i(S^{n-1} \setminus \Lambda)$. The fact that $\tilde{H}^{n-2-i}(S^{n-2}) = 0$ for all $i \geq 1$ implies that $Z$ is acyclic. (For degree zero, note that $Z$ is a connected component.) Since $Z$ is simply connected, we may use the Hurewicz theorem to conclude that the higher homotopy groups vanish as well. Contractibility now follows from Whitehead's theorem.
    %
    %
\end{proof}

For any finitely generated group $G$, there is a family of finite-dimensional,  simplicial complexes on which $G$ acts geometrically, known as Rips complexes.  

\begin{definition}Let $S$ be a finite generating set for $G$ and consider the word metric $d=d_{G,S}$ for $G$ with respect to $S$. For any $r>0$ we define a simplicial complex called the \emph{Rips complex} $P_r(G)=P_r(G,S)$ as follows. The vertex set of $P_r(G)$ is $G$, and $g_0, \ldots, g_k$ span a $k$-simplex if $d(g_i,g_j)\leq r$ for all $i,j\in \{0,\ldots, k\}$.
\end{definition}
The word metric on $G$ induces a metric on $P_r(G)$, and $G$ acts faithfully, cocompactly on $P_r(G)$ with finite stabilizers. By Milnor--Schwarz, $G$ is quasi-isometric to $P_r(G)$. When $G$ is hyperbolic, then for any generating set and $r$ sufficiently large, $P_r(G)$ is in fact contractible \cite{GhysDelaharpe}. In particular, if $G$ is torsion-free then the action of $G$ on $P_r(G)$ is free and thus for $r$ sufficienty large $P_r(G)/G$ is a $K(G,1)$. In this case, Bestvina--Mess show that the proper homotopy type of $P_r(G)$ is an invariant of $G$ \cite{BestvinaMess91}.

\begin{proposition}[Flatness Criterion]\label{prop:Flatness-Criterion}
    If $Z^{+}$ and $Z^{-}$ are simply connected, then the embedding $\Lambda \hookrightarrow \partial_{\infty} G$ is locally flat, hence conjugate by a homeomorphism to a standard embedding $S^{n-2} \hookrightarrow S^{n-1}$.
\end{proposition}

\begin{proof}
    Let $Z$ be either $Z^{+}$ or $Z^{-}$. Since $H$ acts on $\partial_{\infty} G$ and preserves its limit set $\Lambda$, it also preserves the complement $Z^{+} \coprod Z^{-}$. If the action of $H$ does not preserve $Z$, we may find an index 2 subgroup (which we also denote by $H$ and which has the same limit set), whose action preserves $Z$. Since $Z$ simply connected, it is contractible by Lemma \ref{lem:simply.connected.implies.contractibility}. As $H$ is torsion-free, a result of Swenson \cite[Main Theorem (3)]{Swenson} implies that $H$ acts freely, properly discontinuously, and cocompactly on $Z$. In particular, the quotient $\faktor{Z}{H}$ is aspherical and thus a compact $K(H,1)$.   
    
    Since any $K(H,1)$ is unique up to homotopy equivalence, $\faktor{Z}{H}$ is homotopy equivalent to any other $K(H,1)$, notably the quotient of the Rips complex $P_d(H)$ by $H$ for $d\gg 1$. Taking universal coverings, this homotopy equivalence lifts to a proper, $H$-equivariant homotopy equivalence between $Z$ and $P_d(H)$. 

    Choosing a basepoint $z_0\in Z$, the orbit map $H\rightarrow H \cdot z_0$ identifies the boundary at infinity of $Z$ (regarded as a $\delta$-hyperbolic space) with $\Lambda$. Since $Z$ is an ANR, we can apply {\cite[Theorem 2.4]{BartelsLuckWeinberger10}} to obtain that $\Lambda$ is $k$-LCC in $\overline{Z}$. Recalling that $Z$ was either $Z^+$ or $Z^-$, we see that $\Lambda$ is $k$-LCC in both $\overline{Z^{+}}$ and $\overline{Z^{-}}$. By {\cite[Theorem 7.6.5]{DavermanVenema09}}, this implies that $\Lambda$ is locally flat in $\overline{Z^{+}} \cup \overline{Z^{-}} = \partial_{\infty} G$. The generalised Schoenflies theorem now implies that the embedding $\Lambda \hookrightarrow \partial_{\infty} G$ is conjugate by a homeomorphism to a standard embedding $S^{n-2} \hookrightarrow S^{n-1}$ by Remark \ref{rem:schoenflies.implies.standard}.
\end{proof}

Suppose now that $G\cong \pi_1(M)$ for some closed, orientable, aspherical manifold $M$. We now apply the above criterion to give an alternative  characterization of local flatness of $\dinf H$ in terms of the cover of $M$ corresponding to $H$. In this setting, we refer to $H$ as \emph{2-sided} if the action of $H$ on $\dinf G$ preserves both components of $\dinf G\setminus \dinf H$.

\begin{theorem}\label{thm:ProductCover}
     Let  $M$ be a closed, orientable, aspherical $n$-manifold with $G=\pi_1(M)$ hyperbolic, $\dinf G\cong S^{n-1}$ and $n\geq 6$. Suppose $H\leq G$ is a quasi-convex, codimension-1, 2-sided subgroup such that $\dinf H\cong S^{n-2}$. If $C_H\rightarrow M$ denotes the cover associated to $H$,  then $\dinf H\subset \dinf G$ is locally flat if and only if there exists a closed, orientable $(n-1)$-manifold $N$ such that $C_H\cong N\times \R$.
\end{theorem}
\begin{proof} Let $Z^{\pm}$ denote the two complementary components of $\dinf G\setminus \dinf H$.  Since $M$ is orientable and $H$ is 2-sided, the action of $H$ on $\dinf G$ preserves $Z^+$ and $Z^{-}$ separately, acting on each by orientation-preserving homeomorphisms.  The action of $H$ on $\widetilde{V}=\widetilde{M}\cup Z^{+}\cup Z^{-}$ is free, properly discontinuous and cocompact by \cite{Swenson}.  In particular, $p\colon \widetilde{V}\rightarrow V$ is a covering map onto a compact manifold $V$ whose interior is $C_H$ and whose boundary is  $W^+\sqcup W^-$, where $W^\pm=\faktor{Z^\pm} {H}$ are closed and orientable. Since both $W^+$ and $W^-$ are collared  in $V$, we have that $V\simeq \interior(V) = C_H$. The latter is aspherical and has fundamental group isomorphic to $H$. Hence $V$ is a $K(H,1)$. 

 If $\dinf H$ is locally flat, then $Z^+,Z^-$ are both contractible. Since $H$ acts freely properly discontinuously on $Z^+,Z^-$, this implies $W^+$ and $W^-$ are also $K(H,1)$'s. To verify the inclusions $W^+,W^-\hookrightarrow V$ are both homotopy equivalences, it is enough to check that they induce isomorphisms on $\pi_1$. The inclusions are surjective because $p^{-1}(W^+)=Z^+$ and $p^{-1}(W^-)=Z^-$ are both connected.  They are injective because $Z^+$ and $Z^-$ are simply connected. Thus $V$ is an $h$-cobordism. Since $H$ is torsion-free hyperbolic, the Whitehead group $\textrm{Wh}(H)$ vanishes by \cite{BartelsLuck} so $V$ is an $s$-cobordism. Since $\dim(W^+)\geq 5$, we conclude that $V\cong W^+\times [0,1]$. Thus, $C_H = \interior(V)\cong W^+\times \R$ and we can take $N=W^+$.

Conversely, suppose $C_H\cong N\times \R$ for some closed, orientable manifold $N$. Thus the inclusion $N\hookrightarrow V$ is a homotopy equivalence.  Let $V^+\subset V$ be the submanifold bounded by $N\times \{0\}$ and $W^+$, and let $R\cong W^+\times [0,\varepsilon]$ be a collar neighborhood of $W^+$ in  $V^+$. Since $V^+\setminus W^+\cong N\times [0,\infty)$ and $N$ is compact, there exists an injective map $f\colon V^+\rightarrow V^+$ which is the identity on $W^+$, and whose image lies in $R$. We compose $f$ with the projection of $R$ onto $W^+$ to obtain a retraction $r\colon V^+\rightarrow W^+$.  Therefore, the inclusion $\iota\colon W^+\hookrightarrow V^+$ induces an injection of $\pi_1(W^+)$ into $\pi_1(V^+)\cong \pi_1(N)$. On the other hand, since $H\cong \pi_1(N)$ is the group of deck transformations of the cover $Z^+\rightarrow W^+$ we have a short exact sequence:\[1\rightarrow \pi_1(Z^+)\rightarrow \pi_1(W^+)\xrightarrow{\iota_{*}} \pi_1(N)\rightarrow 1.\] Therefore, $\iota_*$ is an isomorphism and $\pi_1(Z^+)=1$. The same argument shows $\pi_1(Z^-)=1$ as well, whence $\dinf H$ is locally flat by Proposition \ref{prop:Flatness-Criterion}.
\end{proof}

\subsection{Embedded submanifolds from codimension-1 subgroups}
Let $M$ be a closed, orientable aspherical $n$-manifold with cubulated hyperbolic fundamental group $G=\pi_1(M)$.  Suppose there exists a quasi-convex subgroup $H\leq G$ such that $H\cong \pi_1(N)$ for some closed orientable aspherical $(n-1)$-manifold $N$. 

\begin{definition}
    A subgroup $H\leq G$ is \emph{square-root closed} if $g^2\in H$ implies $g\in H$.
\end{definition}

\begin{lemma}\label{lem:Sqrt-Amalgamated-Malnormal}
    Let $G$ be a 1-ended torsion-free hyperbolic group. If  $H\leq G$ is malnormal and quasi-convex, then $H$ is square-root closed.  Moreover, if $H$ is 1-ended and has codimension-1, then $G$ splits as an amalgamated product or HNN extension over $H$. 
\end{lemma}
\begin{proof}If  $g^2\in H$ then $\langle g^2\rangle \leq gHg^{-1}\cap H$, and since $g^2$ has infinite order malnormality implies this is only possible if $g\in H$. This proves the first statement. The second statement follows from a result of Kropholler \cite[p.146, Theorem 4.9]{GGTNibloRoller}.
\end{proof}

Our interest in $G$ splitting over a square-root closed subgroup $H$ stems from a theorem of Cappell \cite[Theorem 1]{Cappell} giving homotopical conditions for a manifold $M$ with $\pi_1(M)=G$ to contain an embedded, 2-sided, codimension-1 submanifold $N$ with $\pi_1(N)\cong H$. 

First we introduce some terminology.  Recall that a finite CW-complex is called an orientable \emph{Poincar\'e complex} if its cohomology ring satisfies Poincar\'e duality with respect to the trivial $\Z$-coefficients.  Clearly any space homotopy equivalent to a closed orientable manifold is an orientable Poincar\'e complex. Consider now the following setup.  Let $Y$ be a connected, orientable $(n+1)$-dimensional  Poincar\'e complex and $j\colon X\rightarrow Y$ an embedding of a connected, orientable $n$-dimensional sub-Poincar\'e complex with $n\geq 4$.  Suppose further that  $j_*\colon \pi_1(X)\rightarrow \pi_1(Y)$ is injective. 
\begin{definition}
    Let $X\subset Y$ be as above and let $W$ be a manifold. A homotopy equivalence $f\colon W\rightarrow Y$ is \emph{splittable along $X$} if it is homotopic to a map $g$ which is transverse regular along $X$ (whence $g^{-1}(X)$ is embedded, codimension-1 submanifold of $W$) and if $g$ restricted to both $g^{-1}(X)$ and $g^{-1}(Y\setminus X)$ are homotopy equivalences. If $Y\setminus X$ has two components we require that the restriction of $g$ to each component of $g^{-1}(Y\setminus X)$ be a homotopy equivalence.
\end{definition}

Given a pair $(Y,X)$ of Poincar\'e complexes, a manifold $W$, and a homotopy equivalence $f\colon W\rightarrow Y$, the main result of Cappell \cite[Theorem 1]{Cappell} gives homotopical conditions for $f$ to be splittable. Crucially, one must assume that $\pi_1(X)$ is square-root closed in $\pi_1(Y)$, and that $\pi_1(Y)$ splits over $\pi_1(X)$ either as an amalgamated product or HNN-extension.  The remaining assumptions are $K$-theoretic in nature, and are satisfied vacuously when both $\pi_1(X)$ and $\pi_1(Y)$ are torsion-free hyperbolic. In the setting where $W$ and $X$ are both aspherical manifolds and $\pi_1(X)$ satisfies the Borel conjecture, the conclusion of Cappell's result is that $W$ contains an embedded, 2-sided submanifold homeomorphic to $X$. We now apply Cappell's result to our situation.


\begin{theorem}\label{thm:Finite-Embedded-Cover}
Let  $M$ be a closed, orientable aspherical $n$-manifold with $G=\pi_1(M)$ cubulated hyperbolic and $n\geq 6$. Let $H\leq G$ be a quasi-convex, codimension-1 subgroup such that $H\cong \pi_1(N)$ for some closed aspherical $(n-1)$-manifold $N$.  Then there exist finite covers $ \widehat{M}\rightarrow M$, $ \widehat{N}\rightarrow N$ and an embedding $ \widehat{N}\hookrightarrow \widehat{M}$.
\end{theorem}
\begin{proof}
By Theorem \ref{thm:Markovic-Malnormal}, there exist finite index subgroups $\widehat G\leq G$ and $\widehat H\leq H$ such that $\widehat H\leq \widehat G$ is a malnormal subgroup. Let $\widehat M$ and $\widehat N$ be the corresponding finite covers of $M$ and $N$, respectively. By Lemma \ref{lem:Sqrt-Amalgamated-Malnormal}, $\widehat H$ is square-root closed and $\widehat G$ splits over $\widehat H$ as either an amalgamated product or HNN extension. In other words, $\widehat{G}$ acts on a tree with edge stabilizers conjugate to $\widehat H$ and a single edge orbit. Since $\widehat H$ is quasi-convex, so are the vertex stabilizers \cite[Proposition 1.2]{Bowditch-Canonical-Splittings}. In particular, being torsion-free hyperbolic, they have finite-dimensional classifying spaces (e.g. the quotient of the Rips complex).  

We can then build a finite-dimensional classifying space for $\widehat G$ from classifying spaces for the vertex stabilizers and $\widehat N$.  Explicitly, if $\widehat G=G_0*_{\widehat H}G_1$, then let $X_0$ be a $K(G_0,1)$ and $X_1$ be a $K(G_1,1)$. Now we obtain $X_{\widehat G}$ as the identification space from $X_1$, $X_2$ and $\widehat{N}\times[0,1]$ by gluing $\widehat{N}\times \{i\}$ to $X_i$ by a map inducing the inclusion $\widehat H\hookrightarrow G_i$, for $i=0,1.$ The construction for an HNN extension is similar. $X_{\widehat G}$ is a Poincar\'e duality complex since is it homotopy equivalent to $\widehat M$, and it contains $\widehat N$ as an embedded, 2-sided submanifold. 

Since all groups involved are torsion-free hyperbolic, their Whitehead groups vanish, hence condition (ii) in Theorem 1 of \cite{Cappell} is automatically satisfied. Thus, all hypotheses of Cappell's theorem are satisfied and any  homotopy equivalence $f\colon \widehat M\rightarrow X_{\widehat G}$ is splittable. In particular, we find an embedded, 2-sided, codimension-1 manifold $N'\hookrightarrow \widehat M$ such that $\pi_1(N')$ injects as $\widehat H\leq \widehat G$. Finally, since $\dim(\widehat N)\geq 5$,  we have that $N'$ is homeomorphic to $\widehat N$ by Theorem \ref{thm:Hyperbolic-Borel-Conjecture}.
\end{proof}

\appendix

\section{Proofs of Lemma \ref{lem:refinements.and.stabilizers} and Lemma \ref{lem:inheritance.of.convergence.property}}\label{sec:Appendix}

\begin{lemma}[Lemma \ref{lem:refinements.and.stabilizers}]
    Let $(\mu, K, \mathcal{U})$ be a refinement of $(\mu', K', \mathcal{U}')$ and suppose $U \subset U'$, where $U \in \mathcal{U}$ and $U' \in \mathcal{U}'$. Then $\Stab_{\mu}(U) < \Stab_{\mu'}(U')$.
\end{lemma}

\begin{proof}
    Suppose $g \notin \Stab_{\mu'}(U')$, that is there exists some $V' \in \mathcal{U}'$ such that $\partial U' \neq \partial V'$ and $\mu'(g)(\partial U') = \partial V'$. Since $\mu$ and $\mu'$ coincide on $K'$, we conclude that $\mu(g)(\partial U') = \partial V'$. If $U = U'$, then this implies that $\mu(g)$ does not preserve $\partial U$ and $g \notin \Stab_{\mu}(U)$.
    
    Now suppose $U \subsetneq U'$ and suppose by contradiction that $g \in \Stab_{\mu}(U)$. Since $\mu(g)(\partial U') = \partial V'$ and $\partial U' \neq \partial V'$, there exists some $p_{\infty} \in \partial U'$ such that
    \[ \mu(g)(p_{\infty}) = \mu'(g)(p_{\infty}) \in \partial V' \setminus \partial U'. \]
    Choose a point $p \in U$ and a path $\gamma$ from $p$ to $p_{\infty}$ that meets $\partial U'$ only in its endpoint $p_{\infty}$. (Such a path may be found by exploiting the homeomorphism $\Phi^{U'} : (D^{n}, S^{n-1}) \rightarrow (\overline{U'}, \partial U')$.) The path $\gamma$ passes through a sequence of elements $U = U_0, U_1, U_2, \dots \in \mathcal{U}$. In between these open sets $\gamma$ meets $K$ in a sequence of path-segments $S_0, S_1, \dots \subset K \cap U'$, where $S_i$ is the segment of $\gamma$ between its time in $U_i$ and its time in $U_{i+1}$. We obtain a decomposition of the path $\gamma$ into segments that are contained in $U_0, S_0, U_1, S_1, \dots$ in that order. If this sequence is finite, then it ends in $S_M$ for some $M$. However, this sequence may go on forever. We denote the start and end point of $S_i$ by $p_i$ and $q_i$ respectively.\\
    
    Let $V_i := \mu(g)(U_i)$ and $T_i := \mu(g)(S_i)$. We will prove by induction that $V_i, T_i \subset U'$ for all $i$ (except for the last $T_M$ which, if it exists, will be contained in $\overline{U'}$). The induction starts with $V_0$. Since we assume, by contradiction, that $g \in \Stab_{\mu}(U)$, we have that $V_0 = \mu(g)(U_0) = \mu(g)(U) = U$ which is a subset of $U'$ by assumption.
    
    Now suppose $\mu(g)(U_i) \subset U'$ and suppose $S_i$ is not the last segment in the sequence $U_0, S_0, \dots$. Since $\gamma$ meets $\partial U'$ only in its endpoint and $S_i$ is not the last segment, $p_i \in U'$ and thus $p_i \in \partial U_i \cap U'$. By induction assumption, $\mu(g)(U_i) \subset U'$ and thus $\mu(g)(p_i) \in \mu(g)(\overline{U_i}) \subset \overline{U'}$. Since $p_i \notin K'$, $\partial U' \subset K'$, and $\mu(g)$ sends $K'$ homeomorphically to itself, we conclude that $\mu (g)(p_i) \in U'$.
    
    Since $\gamma$ meets $\partial U'$ only in its endpoint and $S_i$ is not the last segment, we have that $S_i \cap K' = \emptyset$. Since $\mu(g)$ sends $K'$ homeomorphically to itself, we conclude that $T_i \cap K' = \emptyset$. Combined with the fact that the starting point $\mu(g)(p_i)$ of $T_i$ lies in $U'$, we conclude that $T_i \subset U'$.

    We now use this to show that $V_{i+1} \subset U'$. Since $T_i \subset U'$, we have that $\mu(g)(q_i) \in U'$. At the same time, $\mu(g)(q_i) \in \partial U_{i+1}$ and thus $\mu(g)(q_i) \in \mu(g)(\overline{U_{i+1}}) \cap U' = \overline{V_{i+1}} \cap U'$. Since every element in $\mathcal{U}$ is contained in an element of $\mathcal{U}'$ and $\overline{V_{i+1}} \cap U' \neq \emptyset$, we conclude that $V_{i+1} \subset U'$.

    This induction shows that $V_i$ and $T_i$ are contained in $U'$ for all $i$ (except for $T_M$, if it is the last segment of the sequence). Depending on whether the sequence $U_0, S_0, \dots$ is finite or infinite, we now finish the proof in two different ways.
    
    We first deal with the case where the sequence is finite and thus we have a segment $T_M$. By induction, we know that $V_M \subset U'$. By the same argument as before, we see that $p_M$ either lies in $U'$, or is the endpoint of $\gamma$. If it is the endpoint of $\gamma$, then $p_M \in \overline{U_{M}}$ and thus $\mu(g)(p_M) \in \mu(g)(\overline{U_{M}}) \subset \overline{U'}$. Now suppose $p_M$ is not the endpoint of $\gamma$. In that case, $S_M$ is a segment that starts in $U'$, that only meets $\partial U' \subset K'$ in its endpoint, and whose starting point is sent to a point in
    \[ \mu(g)(\overline{U_M} \cap U') \subset \overline{U'} \setminus K' = U'. \]
    We conclude that $T_M$ is contained in $U'$ except for its endpoint which lies in $\partial U'$. Since the endpoint of $T_M$ is the endpoint of $\gamma$, which is $p_{\infty}$, this implies that $\mu(g)(p_{\infty}) \in \partial U'$. However, we chose $p_{\infty}$ so that $\mu(g)(p_{\infty}) \notin \partial U'$, a contradiction.

    Now suppose the sequence $U_0, S_0, \dots$ is infinite. In that case, $V_i, T_i$ are contained in $U'$ for every $i$. Since $\gamma$ is continuous, the endpoints $q_i$ of the segments $S_i$ converge to the point $p_{\infty}$. Since $\mu(g)$ is continuous, we conclude that
    \[ \mu(g)(q_i) \rightarrow \mu(g)(p_{\infty}) \in \partial V' \setminus \partial U'. \]
    However, $\mu(g)(q_i) \in U'$ for all $i$ and thus their limit has to lie in $\overline{U'}$. This is a contradiction. Since we obtain a contradiction in both cases, we conclude that $g$ cannot lie in $\Stab_{\mu}(U) \setminus \Stab_{\mu'}(U')$, which completes the proof.
\end{proof}

\begin{lemma}[Lemma \ref{lem:inheritance.of.convergence.property}]
    Let $(\mu, K, \mathcal{U})$ be a $G$-complex which is a refinement of a convergence $G$-complex $(\mu', K', \mathcal{U}')$. 
    If for every $U' \in \mathcal{U}'$, the group $\mu(\Stab_{\mu'}(U'))$ is a convergence group on $\overline{U'} \cap K$, then $(\mu, K, \mathcal{U})$ is a convergence $G$-complex as well.
\end{lemma}

\begin{proof}
    Let $g_m \in G$ be a sequence of distinct elements. Since $(\mu', K', \mathcal{U}')$ is a convergence $G$-complex, there exist $a, b \in S^{n-1}$ and a subsequence of $g_m$ (which we denote by $g_m$ again) such that $\mu'(g_m) \rightarrow a$ uniformly on compact subsets of $K' \setminus \{ b \}$. This implies that, for the inverse sequence, we have $\mu'(g_m)^{-1} = \mu'(g_m^{-1}) \rightarrow b$ uniformly on compact subsets of $K' \setminus \{ a \}$.

    \subsubsection*{Step 1: Using the convergence action on $K'$ to get parts of the convergence behavior on $K$.}

    Since the actions of $\mu$ and $\mu'$ coincide on $K' \subset K$, we conclude that $\mu(g_m) \rightarrow a$ and $\mu(g_m)^{-1} \rightarrow b$ on compact subsets of $K' \setminus \{ b \}$ and $K' \setminus \{ a \}$ respectively. We thus obtain the following two properties: For every compact set $C \in K \setminus \{ b \}$ and every $\epsilon > 0$, there exists $m'(\epsilon, C) \in \mathbb{N}$ such that
    \begin{equation} \label{eq:forward.contraction}
        \forall m \geq m'(\epsilon, C) : d( \mu(g_m)( C \cap K' ), a) < \epsilon.
    \end{equation}
    Similarly, for every compact set $D \in K \setminus \{ a \}$ and every $\delta > 0$, there exists $m''(\delta, D)$ such that
    \begin{equation} \label{eq:backward.contraction}
        \forall m \geq m''(\delta, D) : d(\mu(g_m^{-1})( D \cap K' ), b) < \delta.
    \end{equation}
    We need to show that for every compact set $C \in K \setminus \{ b \}$ and every $\epsilon > 0$, there exists $m(\epsilon, C)$ such that
    \[ 
    \forall m \geq m(\epsilon, C) : d(\mu(g_m)(C), a) < \epsilon. \]

    \subsubsection*{Step 2: Exhausting $K \setminus \{ b \}$ by well-behaved compact sets.}

    For $\alpha \in (0,1)$, we define
    \[ C_{\alpha} := (K \cap B_{\alpha}(0)) \cup (S^{n-1} \setminus \interior(B_{1-\alpha}(b)) \]
    where $B_r(x)$ denotes the closed ball in $D^n$ of radius $r$ with respect to the euclidean metric centered at $x$ and $\interior(B_r(x))$ its interior.
    
    The collection $C_{\alpha}$ is an exhaustion of $K \setminus \{ b \}$ by compact sets. We may thus assume without loss of generality that $C = C_{\alpha}$ for some $\alpha \in (0, 1)$. Note that, for $1-\alpha$ sufficiently small, $C_{\alpha}$ has the following property: If $C_{\alpha} \cap U' \neq \emptyset$ for some $U' \in \mathcal{U}'$, then either $C_{\alpha} \cap \partial U' \cap \interior(D^{n}) \neq \emptyset$ or $U' = \interior(D^{n})$ and $\partial U' = S^{n-1}$. (This property is crucial to the proof and we highlight that it holds for $D^{n}$ for the same reason as it does for $D^3$.) If there exists $U' \in \mathcal{U}'$ such that $U' = \interior(D^n)$, then $K' = S^{n-1}$ and $(\mu, K, \mathcal{U}$ is a convergence $G$-complex by the assumption in the Lemma. We thus assume from now on that $K' \neq S^{n-1}$. We define compact sets $D_{\alpha}$ that form an exhaustion of $K \setminus \{ a \}$ in the analogous way.\\

    Let $\epsilon > 0$. Since $(K', \mathcal{U}')$ satisfies $(\star)$ (see Remark \ref{rem:condition.star}), there are only finitely many sets $W_1, \dots, W_L \in \mathcal{U}'$ whose diameter is $\geq \frac{\epsilon}{2}$.

    \subsubsection*{Step 3: There exist $M(\alpha), N(\epsilon, \alpha) \in \mathbb{N}$ and sets $V_j \in \mathcal{U}'$, where $j \in \{ 1, \dots, M(\alpha) \}$ such that the following holds:} If $V \in \mathcal{U}'$ satisfies
        \begin{itemize}
            \item $\overline{V} \cap C_{\alpha} \neq \emptyset$,
            \item there exists $m > N(\epsilon, \alpha)$ such that $\mu(g_m)(V) = W_i$ for some $1 \leq i \leq L$,
        \end{itemize}
    then $V = V_j$ for some $j$.\\

    By construction, $C_{\alpha}$ does not intersect the open ball in $D^{n}$ of radius $1-\alpha$ centered at $b$. Put $\delta := 1 - \alpha$. We obtain the sets $V_i$ as follows: Since $(K', \mathcal{U}')$ satisfies $(\star)$, there exist only finitely many sets $V_1, \dots, V_M \in \mathcal{U}'$ with diameter $\geq \frac{\delta}{2}$. We set the number of these sets to be $M = M(\alpha)$.

    Choose $\alpha'$ such that $D_{\alpha'}$ intersects each $W_i$ (thus $\alpha' = \alpha'(\epsilon)$, as the $W_i$ depend only on $\epsilon$). Furthermore, choose $\alpha'$ large enough such that, whenever $D_{\alpha'} \cap U' \neq \emptyset$, then $D_{\alpha'} \cap \partial U' \cap \interior(D^{n}) \neq \emptyset$. (This depends only on the generalized cell decomposition $(K', \mathcal{U}')$ and the fact that we already dealt with the case $K' = S^{n-1}$.) Set $D := D_{\alpha'}$. By equation (\ref{eq:backward.contraction}), we know that for every $m > m''(\frac{\delta}{2}, D)$, we have
    \[ d(\mu(g_m^{-1})( D \cap K' ), b) < \frac{\delta}{2}. \]
    We set $N(\epsilon, \alpha) := m''(\frac{\delta}{2}, D)$, which depends on $\alpha$ and $\epsilon$, as $\delta$ and $D$ are determined by $\alpha$ and $\epsilon$ respectively. We observe that, since $D$ intersects each $W_i$ by assumption and $\alpha'$ was chosen sufficiently large, it intersects $\partial W_i \cap \interior(D^{n})$. Thus we find $p_i \in \partial W_i \cap D$. Since $\partial W_i \subset K'$, equation (\ref{eq:backward.contraction}) implies for every $m > N(\epsilon, \alpha)$ that
    \begin{equation} \label{eq:bounding.p_i}
    d(\mu(g_m^{-1})(p_i), b) < \frac{\delta}{2}.
    \end{equation}

    We claim that these choices of $M(\alpha)$, $N(\epsilon, \alpha)$, $V_1, \dots, V_{M(\alpha)}$ have the properties we require. Let $V \in \mathcal{U}'$ such that $\mu(g_m)(V) = W_i$ for some $i$ and some $m > N(\epsilon, \alpha)$. Suppose $V \neq V_j$ for all $1 \leq j \leq M(\alpha)$. Then
    \[ \frac{\delta}{2} > \diam(V) = \diam( \mu(g_m^{-1})(W_i) ) = \diam( \mu( g_m^{-1} ) ( \overline{W_i} ) ). \]
    Combined with equation (\ref{eq:bounding.p_i}), this implies that
    \[ d(\overline{V}, b) = d( \mu(g_m^{-1})( \overline{ W_i }), b) < \delta = 1 - \alpha. \]
    Therefore, $\overline{V}$ cannot intersect $C_{\alpha}$ which proves Step 3.\\

    \subsubsection*{Step 4: Produce a subsequence of $(g_m)_m$ that shrinks the diameter of $\partial V$ for all but finitely many $V \in \mathcal{U}'$ that intersect $C_{\alpha}$.}

    Choose some $V \in \mathcal{U}'$ such that $C_{\alpha} \cap V \neq \emptyset$. We partition the sequence $(g_m)_m$ into $L+1$ new sequences $g_{m}^{i}$ for $0 \leq i \leq L$ as follows: The sequence $g_m^{0}$ consists of all elements of $(g_m)_m$ such that
    \[ \diam(\mu(g_m^{0})(\partial V) ) \leq \frac{\epsilon}{2}. \]
    For $i \geq 1$, $(g_m^{i})_m$ consists of all elements of $(g_m)_m$ that satisfy $\mu(g_m^{i})(V) = W_i$.

    It is important to note that we index the elements $g_m^{i}$ such that $g_m^{i} = g_m$ for every $m$. In turn, this means that $g_m^{i}$ is defined only for certain $m \in \mathbb{N}$. It is important that we index our elements in this way to formulate the following statement.

    We claim that, if $V$ is not one of the sets $V_j$ from Step 3, then $g_m^{i}$ is not defined for any $1 \leq i \leq L$ and any $m > N(\epsilon, \alpha)$. Indeed, if $V$ is not any of the sets $V_j$ and $m > N(\epsilon, \alpha)$, then $V$ cannot satisfy both properties required in Step 3. But $C_{\alpha} \cap V \neq \emptyset$ by assumption on $V$ and thus $\mu(g_m)(V) \neq W_i$ for all $1 \leq i \leq L$.\\ 

    \subsubsection*{Step 5: Estimate $d(\mu(g_m^{0})(C_{\alpha} \cap V),a)$ for $V \in \mathcal{U}'$.}

    Let $V \in \mathcal{U}'$ such that $C_{\alpha} \cap V \neq \emptyset$. Since $C_{\alpha}$ intersects $V$ and $\alpha$ was chosen sufficiently large, there exists $q \in \partial V \cap C_{\alpha} \cap \interior(D^{n})$. 
    Applying inequality (\ref{eq:forward.contraction}), we obtain
    \[ \forall m > m'\left(\frac{\epsilon}{2},C_{\alpha} \right) : d(\mu(g_m)(q), a) < \frac{\epsilon}{2}. \]
    Restricting to the elements $g_m^{0}$, we also have
    \[ \diam\left(\mu\left(g_m^{0}\right)(\overline{V})\right) \leq \frac{\epsilon}{2}. \]
    Combining these two inequalities, we obtain
    \begin{equation} \label{eq:convergence.inequality}
    \forall m > m'\left(\frac{\epsilon}{2},C_{\alpha} \right) : d(\mu(g_m^{0})( C_{\alpha} \cap V ),a) < \epsilon.
    \end{equation}
    This is the equality we desire for all $g_m$ (and all $m$ greater than some $m_0$). We are left to show that this inequality also holds for the elements $g_m^{i}$ when $m$ is sufficiently large.
    
    If $V \neq V_j$ for all $1 \leq j \leq L$, then $g_m^{i}$ is not defined for any $m > N(\epsilon, \alpha)$ by Step 4 and we are done.

    Let $1 \leq j \leq L$ and suppose $V = V_j$. Let $m_i$ be the smallest $m$ such that $g_m^{i}$ is defined. We define $h_m^{i} := \left( g_{m_i}^{i} \right)^{-1} \cdot g_m^{i}$ and compute
    \[ \mu(h_m^{i}) = \mu(g_m^{i}) \circ \mu(g_{m_i}^{i})^{-1}. \]
    (Recall our convention that $\mu(gh) = \mu(h) \circ \mu(g)$.) We conclude that $\mu(h_m^{i})(W_i) = W_i$ and thus $h_m^{i} \in \Stab_{\mu'}(W_i)$. Since $\mu(g_m) \rightarrow a$ uniformly on compact sets in $K' \setminus \{ b \}$, we have that $\mu(h_m^{i}) \rightarrow a$ uniformly on compact sets in $K' \setminus \{ b_i \}$, where $b_i := \mu(g_{m_i}^{i})(b)$.

    Suppose, $g_m^{i}$ did not converge to $a$ uniformly on $C_{\alpha} \cap V_j$. Then there exists an infinite subsequence $(h_{m_k}^{i})_k$ and a sequence of points $z_{m_k} \in \mu(g_{m_i}^{i})(C_{\alpha}) \cap W_i$ such that
    \[ \forall m : d(\mu(h_{m_k}^{i})(z_{m_k}), a) \geq \epsilon. \]
    We first observe that $(z_{m_k})_k$ cannot converge to $b_i$. Indeed, if $z_{m_k} \rightarrow b_i$, then $\mu(g_{m_i}^i)^{-1}(z_{m_k}) \rightarrow b$ while $\mu(g_{m_i}^i)^{-1}(z_{m_k}) \in C_{\alpha}$ for every $k$. Since $C_{\alpha}$ is closed by construction, this implies that $b \in C_{\alpha}$. But $b \notin C_{\alpha}$ so this is a contradiction and we conclude that $(z_{m_k})_k$ cannot converge to $b$.
    
    By assumption, $\mu(\Stab_{\mu'}(W_i))$ is a convergence group on $\overline{W_i} \cap K$. Therefore, we may pass to a subsequence of $(h_{m_k}^{i})_k$ (which we denote by $(h_{m_k}^i)_k$ again) and find points $a', b' \in \overline{W_i} \cap K$ such that $\mu(h_{m_k}^{i}) \rightarrow a'$ uniformly on compact subsets of $(\overline{W_i} \cap K) \setminus \{ b' \}$.

    On the other hand, we know that $\mu(h_m^{i}) \rightarrow a$ uniformly on compact sets in $K' \setminus \{ b_i \}$, in particular on compact sets in $\partial W_i \setminus \{ b_i \}$. Since $\partial W_i \approx S^{n-1}$ contains infinitely many points, uniform convergence on compact sets in $\partial W_i \setminus \{ b', b_i \}$ yields that $a = a'$ and $b_i = b'$. Therefore, $\mu(h_{m_k}^{i}) \rightarrow a$ uniformly on compact subsets of $(\overline{W_i} \cap K) \setminus \{ b_i \}$. Since $z_{m_k} \in W_i \cap K$ and $z_{m_k}$ does not converge to $b_i$, we conclude that there exists $m_i(\epsilon, C_{\alpha}, V_j)$ such that
    \[ \forall m > m_i(\epsilon, C_{\alpha}, V_j) : d(\mu(h_{m_k}^{i}(z_{m_k}), a) < \epsilon. \]
    This is a contradiction and we conclude that there exists $m_i(\epsilon, C_{\alpha}, V_j)$ such that
    \[ \forall m > m_i(\epsilon, C_{\alpha}, V_j) : d(\mu(g_{m}^{i}(C_{\alpha} \cap V_j), a) \leq \epsilon. \]

    \subsubsection*{Step 6: Concluding the Lemma.}

    Set
    \[ m(\epsilon, \alpha) := \max\left\{ m'\left(\frac{\epsilon}{2}, C_{\alpha} \right), N(\epsilon, \alpha), m_1(\epsilon, C_{\alpha}, V_1), \dots, m_{M(\alpha)}(\epsilon, C_{\alpha}, V_{M(\alpha)}) \right\}. \]

    By the estimates obtained in Step 5, we have for every $V \in \mathcal{U}$ such that $C_{\alpha} \cap V \neq \emptyset$ that
    \[ \forall m \geq m(\epsilon, \alpha) : d( \mu( g_m( C_{\alpha} \cap V), a ) \leq \epsilon. \]
    Combined with inequality (\ref{eq:forward.contraction}) this implies that $(\mu, K, \mathcal{U})$ is a convergence $G$-complex, which completes the proof.    
\end{proof}

\bibliography{mybibmini}
\bibliographystyle{alpha}

\end{document}